\documentclass{amsart}

\usepackage{amsmath}
\usepackage{amsthm}
\usepackage{amssymb}
\usepackage{xypic}

\usepackage{tikz}
\usepackage{graphicx}
\usepackage{here}
\usetikzlibrary{decorations,decorations.pathmorphing}
\usetikzlibrary{patterns}

\theoremstyle{plain}
\newtheorem{theorem}{Theorem}[section]
\newtheorem{proposition}[theorem]{Proposition}
\newtheorem{corollary}[theorem]{Corollary}
\newtheorem{lemma}[theorem]{Lemma}
\newtheorem{claim}{Claim}
\newtheorem{problem}[theorem]{Problem}

\theoremstyle{definition}
\newtheorem{definition}[theorem]{Definition}
\newtheorem{notation}[theorem]{Notation}

\theoremstyle{remark}
\newtheorem{remark}[theorem]{Remark}
\newtheorem{example}[theorem]{Example}

\newcommand{\coeff}{\mathop{\mathrm{coeff}}\nolimits}
\newcommand{\codim}{\mathop{\mathrm{codim}}\nolimits}

\newcommand{\Newt}{\mathop{\mathrm{Newt}}\nolimits}

\newcommand{\ord}{\mathop{\mathrm{ord}}\nolimits}

\newcommand{\dist}{\mathop{\mathrm{dist}}\nolimits}
\newcommand{\val}{\mathop{\mathrm{val}}}
\newcommand{\trop}{\mathop{\mathrm{trop}}}
\newcommand{\Trop}{\mathop{\mathrm{Trop}}}
\newcommand{\Supp}{\mathop{\mathrm{Supp}}}
\newcommand{\Aff}{\mathop{\mathrm{Aff}}}
\newcommand{\tropin}{\mathop{\mathrm{trop}^{-1}}}
\newcommand{\Hf}{\mathop{\mathrm{H}_4(\lambda; f, g; L)}}
\newcommand{\Elim}{\mathop{\mathrm{Elim}(\lambda; f, g; L)}}

\newcommand{\cPI}{\mathcal{PI}}
\newcommand{\cRayfg}{\mathcal{R}_1(f, g)}
\newcommand{\cSIfg}{\mathcal{L}_{\mathrm{s}}(f, g)}
\newcommand{\cLSfg}{\mathcal{LS}_2(f, g)}
\newcommand{\cRay}{\mathcal{R}_1}
\newcommand{\cSI}{\mathcal{L}_{\mathrm{s}}}
\newcommand{\cLS}{\mathcal{LS}_2}
\newcommand{\cF}{\mathcal{F}}
\newcommand{\cG}{\mathcal{G}}
\newcommand{\cRayFG}{\mathcal{R}_1(\mathcal{F}, \mathcal{G})}
\newcommand{\cSIFG}{\mathcal{L}_{\mathrm{s}}(\mathcal{F}, \mathcal{G})}
\newcommand{\cLSFG}{\mathcal{LS}_2(\mathcal{F}, \mathcal{G})}

\newcommand{\Mm}{\mathrm{M}_-}
\newcommand{\Mp}{\mathrm{M}_+}

\newcommand{\Sone}{\Sigma^{(1)}}

\newcommand{\Sn}{\Sigma^{(n)}}

\newcommand{\sembunioil}{\overline{\mathbf{i}_0\mathbf{i}_1}}
\newcommand{\sembunjojl}{\overline{\mathbf{j}_0\mathbf{j}_1}}
\newcommand{\sembunij}{\overline{\mathbf{i}\mathbf{j}}}
\newcommand{\sembunpipj}{\overline{p_ip_j}}
\newcommand{\sembunpqpr}{\overline{p_qp_r}}
\newcommand{\sembunpspt}{\overline{p_sp_t}}

\newcommand{\sembunab}{\overline{ab}}

\newcommand{\sembunqiqil}{\overline{q_{i}q_{i+1}}}

\newcommand{\bio}{\mathbf{i}_0}
\newcommand{\bil}{\mathbf{i}_1}
\newcommand{\bi}{\mathbf{i}}
\newcommand{\bj}{\mathbf{j}}
\newcommand{\bl}{\mathbf{l}}
\newcommand{\bim}{\mathbf{i}_-}
\newcommand{\bip}{\mathbf{i}_+}
\newcommand{\bjo}{\mathbf{j}_0}
\newcommand{\bjl}{\mathbf{j}_1}
\newcommand{\bjp}{\mathbf{j}_+}
\newcommand{\bjm}{\mathbf{j}_-}
\newcommand{\blo}{\mathbf{l}_0}
\newcommand{\bll}{\mathbf{l}_1}
\newcommand{\blp}{\mathbf{l}_+}
\newcommand{\blm}{\mathbf{l}_-}

\newcommand{\de}{\lambda}

\title[Tropical lifting problem for the intersection of plane curves]
{Tropical lifting problem for the intersection of plane curves}

\author{Masayuki Sukenaga}
\address{
Department of Mathematics, Graduate School of Science, Hiroshima University, 
1-3-1 Kagamiyama, Higashi-Hiroshima, 739-8526 JAPAN}
\email{d215394@hiroshima-u.ac.jp}



\begin{document}

\begin{abstract}
Given a tropical divisor $D$ in the intersection of two tropical plane curves, we study when it can be realized as the tropicalization of the intersection of two algebraic curves, and give a sufficient condition.
It is shown that under a certain condition involving a graph determined by these tropical curves, we can algorithmically find algebraic curves such that the tropicalization of their intersection is $D$. 
\end{abstract}

\maketitle

\section{Introduction}

In this paper, let $k$ be a fixed algebraically closed field with a nontrivial valuation $\val: k\to \mathbb{R}\cup \{+\infty\}$.
A tropical plane curve is obtained by the tropicalization of an algebraic curve in $(k^*)^2$.
Here, the tropicalization is defined using the following map:
\begin{eqnarray*}
\trop: (k^*)^2&\to&{\mathbb{R}}^2\\
(x, y)&\mapsto&(-\val(x), -\val(y)).
\end{eqnarray*}
Let $f=\sum_{ij}c_{ij}x^iy^j \in k[x^{\pm1}, y^{\pm1}]$ be given.
For a given tropical divisor $D$ on the tropical plane curve $\Trop(V(f))$, it has been considered whether $D$ can be obtained by the tropicalization of the intersection of two algebraic curves (\cite{BL}, \cite{LS}, \cite{M}, \cite{OP} and \cite{OR}).
This kind of problem is called a tropical lifting problem or a tropical realization problem.
In this paper, we give a sufficient condition involving a graph determined by given tropical curves for the lifting problem for the intersection of curves.

\subsection{Tropical lifting problem}

First, we explain what is known about tropical lifting problems for the intersection of two tropical plane curves.
Let $\cF$ and $\cG$ be bivariate tropical polynomials.
They define the tropical plane curves $V(\cF)$ and $V(\cG)$ (see Section 2).

\begin{definition}
We say that two tropical plane curves $\Gamma_1$ and $\Gamma_2$ meet \textit{properly} at a point $p$ if $p$ is an isolated point in $\Gamma_1\cap \Gamma_2$.
We define $\cPI(\cF, \cG)$ as the multiset of the points $p$ at which $V(\cF)$ and $V(\cG)$ meet properly, with the local intersection numbers as multiplicities.
We also write $\cPI(\trop(f), \trop(g))$ as $\cPI(f, g)$ (for the tropicalization of a Laurent polynomial, see Definition \ref{TLP}).
\end{definition}

Proper intersections are the simplest intersections of tropical plane curves.
Tropical lifting problems of proper intersections are studied in \cite{OR} (see Theorem \ref{OR6.13}).
For algebraic curves $C_1, C_2\subset (k^*)^2$, if the tropical curves $\Trop(C_1)$ and $\Trop(C_2)$ meet properly, then $\trop(C_1\cap C_2)$ is equal to the intersection $\Trop(C_1)\cap \Trop(C_2)$, considered with multiplicities.
Thus, we have to consider the case where $\Trop(C_1)\cap \Trop(C_2)$ does not consist of isolated points, i.e., contains $1$-dimensional components.

\begin{definition}
A (tropical) \textit{divisor} on a tropical curve $\Gamma$ is a finite sum $D=\sum n_iP_i$, where $P_i\in \Gamma$ and $n_i\in \mathbb{Z}$.
\end{definition}

\begin{definition}
A \textit{tropical rational function} on a tropical curve $\Gamma$ is a continuous function $\psi: \Gamma\to \mathbb{R}$ such that its restriction to any edge of $\Gamma$ is a piecewise linear function with integer slopes, i.e., piecewise $\mathbb{Z}$-affine, and with only finitely many pieces.
The divisor of $\psi$ is $\sum_{P\in \Gamma}\ord_P(\psi)P$, where $\ord_P(\psi)$ is $(-1)$ times the sum of the outgoing slopes of $\psi$ at $P$.
We write $(\psi)$ for the divisor of $\psi$.
If $D$ and $E$ are divisors such that $D-E=(\psi)$ for some tropical rational function $\psi$, we say that $D$ and $E$ are linearly equivalent.
We define the support of $\psi$ as $\Supp(\psi)=\overline{\{P\in \Gamma\ |\ \psi(P)\neq 0\}}$.
\end{definition}

Morrison showed the following necessary condition for the realizability of a tropical divisor as the intersection of curves.

\begin{theorem}\label{M1.2}
\cite[Theorem 1.2]{M}
Let $\Gamma_1$ and $\Gamma_2$ be tropical plane curves such that $\Gamma_1$ is smooth (Definition \ref{STC}).
Let $E$ be the stable intersection divisor (Definition \ref{SID}) of $\Gamma_1$ and $\Gamma_2$, and let $D=\sum n_iP_i\ (n_i\in \mathbb{Z}_{\geq 0})$ be a divisor on $\Gamma_1\cap \Gamma_2$.
Assume that there exist algebraic curves $C_1, C_2\subset (k^*)^2$ without common irreducible components such that $\Trop(C_1)=\Gamma_1$, $\Trop(C_2)=\Gamma_2$, and  $\trop(C_1\cap C_2)=D$ as multisets.
Then, there exists a tropical rational function $\psi$ on $\Gamma_1$ such that $(\psi)=D-E$ and $\Supp(\psi)\subset \Gamma_1 \cap \Gamma_2$.
\end{theorem}

In \cite{M}, a conjecture on the converse is also presented.

\begin{problem}\label{conjM}
\cite[Conjecture 3.3]{M}
Let $\psi$ be a tropical rational function on a tropical curve $\Trop(V(f))$ such that $\Supp(\psi)\subset \Trop(V(f)) \cap \Trop(V(g))$ and $(\psi)=D-E$, where $E$ is the stable intersection divisor and $D=\sum n_iP_i\ (n_i\in \mathbb{Z}_{\geq 0})$ is a divisor on $\Gamma_1\cap \Gamma_2$ such that each coordinate of $P_i$ is in the value group of $k$.
Then is it possible to find $f', g'\in k[x^{\pm1}, y^{\pm1}]$ such that $\Trop(V(f'))=\Trop(V(f))$, $\Trop(V(g'))=\Trop(V(g))$ and $\trop(V(f', g'))=D$?
\end{problem}

This was answered in the negative.
See \cite[Theorem 5.2]{LS} for a tropical self-intersection case, and \cite[Lemma 3.15]{BL} for a non self-intersection case.
On the other hand, it would be useful to find sufficient conditions for the realizability.
The purpose of this paper is to give a sufficient condition involving a certain graph.
We introduce several notations before explaining the setting of the main problem.

\begin{definition}
Let $\Gamma_1$ and $\Gamma_2$ be tropical plane curves.
Let $\mathfrak{K}$ be a connected component of $\Gamma_1 \cap \Gamma_2$.
The \textit{intersection multiplicity} of $\Gamma_1 \cap \Gamma_2$ on $\mathfrak{K}$ is defined as the sum of the multiplicities of the stable intersection points on $\mathfrak{K}$ (Definitions \ref{IM} and \ref{SID}).
\end{definition}

Let us introduce notations on the second simplest components of the intersection of tropical curves.

\begin{definition}\label{RLS}
We define $\cRayFG$ as the set of rays $L$ satisfying the following:
\begin{itemize}
\item $L$ is a connected component of the intersection $V(\cF) \cap V(\cG)$.
\item The intersection multiplicity of $V(\cF)$ and $V(\cG)$ on $L$ is $1$.
\item Each $1$-dimensional cell of $V(\cF)$ or $V(\cG)$ which has a $1$-dimensional intersection with $L$ and contains the endpoint of $L$ as its vertex has weight $1$.
\end{itemize}
Also, we define $\cLSFG$ as the set of (bounded) line segments $L$ satisfying the following:
\begin{itemize}
\item $L$ is a connected component of the intersection $V(\cF) \cap V(\cG)$.
\item The intersection multiplicity of $V(\cF)$ and $V(\cG)$ on $L$ is $2$.
\item Each $1$-dimensional cell of $V(\cF)$ or $V(\cG)$ which has a $1$-dimensional intersection with $L$ and contains an endpoint of $L$ as its vertex has weight $1$.
\end{itemize}

We write $\cSIFG:=\cRayFG \cup \cLSFG$.
It turns out that any edge of $V(\cF)$ or $V(\cG)$ that meets $L\in \cSIFG$ has weight $1$ and any vertex contained in $L$ is smooth (see Lemma \ref{L}).
For Laurent polynomials $f, g\in k[x^{\pm1}, y^{\pm1}]$, we also write $\cRayfg$, $\cLSfg$ and $\cSIfg$ for $\cRay(\trop(f), \trop(g))$, $\cLS(\trop(f), \trop(g))$ and $\cSI(\trop(f), \trop(g))$, respectively.
\end{definition}

Thus, the connected components of $V(\cF)\cap V(\cG)$ are points in $\cPI(\cF, \cG)$, elements of $\cSIFG$, and possibly a number of other $1$-dimensional sets.

We will see that, if $L\in \cRayfg$, then there are at most one point in the intersection $\trop(V(f, g))\cap L$ (Corollary \ref{raynihaikko}).
Thus, in this paper, we will consider the following condition.

\begin{definition}\label{star}
The condition $(*)$ on a divisor $D$ on $\Trop(V(f)) \cap \Trop(V(g))$ is the following:
\begin{itemize}
\item
$D=\sum n_iP_i$\ ($n_i\geq 0$).
\item
Each coordinate of $P_i$ is in the value group of $k$.
\item
There exists a tropical rational function $\psi$ on the tropical curve $\Trop(V(f))$ such that $\Supp(\psi)\subset \Trop(V(f)) \cap \Trop(V(g))$ and $(\psi)=D-E$, where $E$ is the stable intersection divisor of $\Trop(V(f))$ and $\Trop(V(g))$.
\item
For $L\in \cRayfg$, $\deg(D|_L)=1$.
\end{itemize}
\end{definition}

Note that this condition is natural in view of Theorem \ref{M1.2}.

\begin{notation}
For a tropical plane curve $\Gamma$, we write $\Sn(\Gamma)$ for the set of the $n$-dimensional cells of $\Gamma$ (see Theorem \ref{dualsubdiv}).
For a tropical polynomial $\cF \in \mathbb{T}[x^{\pm1}, y^{\pm1}]$, we write $\Delta^{(n)}_{\cF}$ for the set of the $n$-dimensional cells of $\Delta_{\cF}$, where $\Delta_{\cF}$ is the dual subdivision of the Newton polygon of $\cF$ (see Definition \ref{dualsubdiv}).
For a Laurent polynomial $f\in k[x^{\pm1}, y^{\pm1}]$, we write $\Delta_f$ and $\Delta^{(n)}_f$ for $\Delta_{\trop(f)}$ and $\Delta_{\trop(f)}^{(n)}$, respectively.
\end{notation}

Note that the intersection multiplicity at an endpoint of a ray or a line segment $L\in \cSIfg$ must be at least $1$, and hence the tropical curves  $\Trop(V(f))$ and $\Trop(V(g))$ have no vertices in the interior of $L$ (see Lemma \ref{L} for details).
Thus, we can define the following maps.

\begin{definition}
We define maps $\phi_i$ ($i=1, 2$) as follows:
\begin{eqnarray*}
\phi_1: \cSIFG&\to&\Sone(V(\cF))\\
L&\mapsto&\text{the $1$-dimensional cell of $V(\cF)$ containing $L$},\\
\phi_2: \cSIFG&\to&\Sone(V(\cG))\\
L&\mapsto&\text{the $1$-dimensional cell of $V(\cG)$ containing $L$},
\end{eqnarray*}
and we define maps $\Phi_i$ ($i=1, 2$) as follows:
\begin{eqnarray*}
\Phi_1: \cSIFG&\to&\Delta_{\cF}^{(1)}\\
L&\mapsto&\text{the $1$-simplex of $\Delta_{\cF}$ corresponding to $\phi_1(L)$},\\
\Phi_2: \cSIFG&\to&\Delta_{\cG}^{(1)}\\
L&\mapsto&\text{the $1$-simplex of $\Delta_{\cG}$ corresponding to $\phi_2(L)$}.
\end{eqnarray*}
\end{definition}

\begin{notation}
Let $a, b\in \mathbb{R}^2$ ($a\neq b$) be points such that the line segment $\sembunab$ has a rational slope.
Then, there is a primitive integer vector $v\in \mathbb{Z}^2$ which has the same slope as $\sembunab$.
The \textit{lattice length} of $\sembunab$ is the ordinary length of $\sembunab$ divided by the ordinary length of $v$.
When $a=b$, we define the lattice length of $\sembunab$ as $0$.
We write $\dist(a, b)$ for the lattice length of $\sembunab$.
We note that $\dist$ does not satisfy the metric inequality.
\end{notation}

On a line segment $L\in \cLS(\cF, \cG)$, a divisor $D$ satisfying $(*)$ can be described as follows.

\begin{lemma}\label{delta_D,L}
Let $L\in \cLSFG$ be a line segment.
Let $D$ be a divisor satisfying $(*)$.
Then, $D|_L=P_1+P_2$ for some $P_1, P_2\in L$, and we have $\dist(P_+, P_1)= \dist(P_-, P_2)$ and $\dist(P_+, P_2)= \dist(P_-, P_1)$, where $P_+$ and $P_-$ are the endpoints of $L$.
\end{lemma}

\begin{proof}
Straightforward from the fact that a tropical rational function $\psi$ on $V(\cF)$ as in $(*)$ takes $0$ at $P_+$ and $P_-$.
\end{proof}

\begin{notation}
Let a tropical divisor $D$ satisfy $(*)$.
For a line segment $L\in \cLSfg$, we define $\dist(D|_L, E|_L)=\min\{\dist(P_+, P_1), \dist(P_+, P_2)\}$ using the notation of Lemma \ref{delta_D,L}.
Also, when $L\in \cRay(f, g)$, we write $\dist(D|_L, E|_L)$ for the lattice length of the distance of the point in $D|_L$ and the endpoint of $L$.
\end{notation}

By analogy with plane algebraic curves, it is a natural setting to fix $f$ and change $g$ in realizing $D$, i.e., the zeros of $\psi$.
For example, in \cite{LS}, a tropical curve $\Gamma$ and an algebraic curve $C$ satisfying $\Trop(C)=\Gamma$ are fixed and $\trop(C\cap C')$ are studied for curves $C'$ with $\Trop(C')=\Gamma$.
Also, it would be useful to study whether it is possible to realize a certain part of $D$.
Let $\cSI'$ be a subset of $\cSI(f, g)$ and $\cPI:=\cPI(f, g)$ the proper intersections.
Let $D|_{\cSI'\cup \cPI}$ denote the restriction of $D$ to the union of the elements of $\cSI'\cup \cPI$.
Then, when is it possible to realize $D|_{\cSI'\cup \cPI}$, i.e., does there exist a Laurent polynomial $g'\in k[x^{\pm1}, y^{\pm1}]$ such that $\Trop(V(g'))=\Trop(V(g))$ and $\trop(V(f, g'))|_{\cSI' \cup \cPI}=D|_{\cSI'\cup \cPI}$?

\subsection{Main result}

As a partial answer to the above question, our main theorems give sufficient conditions for the realizability.
To state the main theorems, we introduce terminologies on trees.

\begin{notation}
It is well known that any two vertices of a tree $T$ are connected by a unique simple path in $T$ (see \cite[Theorem 1.5.1]{D}). 
We write $pTq$ for the simple path between two vertices $p$ and $q$ in $T$.
\end{notation}

\begin{definition}
Let $T$ be a tree and $\leq$ a total ordering on the set of its vertices.
Let $p_0$ denote the smallest vertex for $\leq$.
The order $\leq$ is called \textit{normal} if $p\in p_0Tq$ implies $p\leq q$.
\end{definition}

\begin{definition}\label{mu}
For lattice points $\bi, \bj \in \mathbb{Z}^2$ such that $\bj-\bi$ is primitive and a tropical polynomial $\mathcal{F}=\bigoplus_{\mathbf{i}\in \mathbb{Z}^2} \alpha_{\mathbf{i}} \mathbf{x}^{\bi}$ with $\alpha_{\bi}, \alpha_{\bj}\neq -\infty$, where $\mathbf{x}^{(i_1, i_2)}$ denotes $x^{i_1}y^{i_2}$, we define $\mu_n(\cF; \sembunij)$ ($n\in \mathbb{Z}$) and $\mu(\cF; \sembunij)$ by
\begin{eqnarray*}
\mu_n(\cF; \sembunij)&:=&-\alpha_{\bi+n(\bj-\bi)}+\alpha_{\bi}+n(\alpha_{\bj}-\alpha_{\bi}),\\
\mu(\cF; \sembunij)&:=&\min\{\mu_n(\cF; \sembunij)\ |\ n\in \mathbb{Z}\setminus \{0, 1\}\}.
\end{eqnarray*}
For $f=\sum_{\bi}c_{\bi}\mathbf{x}^{\bi}\in k[x^{\pm1}, y^{\pm1}]$ with $c_{\bi}, c_{\bj} \neq 0$, we write $\mu_n(f; \sembunij)$ and $\mu(f; \sembunij)$ for $\mu_n(\trop(f); \sembunij)$ and $\mu(\trop(f); \sembunij)$, respectively.
\end{definition}

\begin{figure}[H]
\centering
\begin{tikzpicture}
\draw[->,>=stealth] (-0.8,0)--(2,0) node[right] {$n$};
\draw[->,>=stealth] (0,-0.75)--(0,2) node[above] {$\alpha_{\bi+n(\bj-\bi)}$};
\draw[<->,>=stealth] (-0.5,-0.225)--(-0.5,0.75);
\draw[<->,>=stealth] (1,1.025)--(1,1.5);
\coordinate [label=below:\text{$\mu_{-1}(\cF; \sembunij)$}] (L6) at (-1.35,0.6);
\coordinate [label=below:\text{$\mu_2(\cF; \sembunij)=\mu(\cF; \sembunij)$}] (L7) at (2.6,1.7);
\coordinate [label=below:\text{$0$}] (L6) at (-0.15,0);
\coordinate [label=below:\text{$1$}] (L6) at (0.5,0);
\coordinate [label=below:\text{$2$}] (L6) at (1,0);
\coordinate [label=below:\text{$3$}] (L6) at (1.5,0);
\coordinate (p1) at (-0.5,-0.25);
\coordinate (p2) at (0,1);
\coordinate (p3) at (0.5,1.25);
\coordinate (p4) at (1,1);
\coordinate (p5) at (1.5,0.5);
\coordinate (p6) at (0.1,0);
\coordinate (p7) at (0.7,-0.6);
\coordinate (p8) at (0.9,-0.6);
\coordinate (p9) at (0.9,-0.8);
\coordinate (p10) at (0.7,-1);
\coordinate (p11) at (2,-0.8);
\coordinate (p12) at (0.3,-1.8);
\coordinate (p13) at (0.1,-2);
\coordinate (p14) at (-1.4,-2);
\coordinate (p15) at (-1.4,-0.8); 
\coordinate (p16) at (-1.6,-2.2); 
\draw (p1)--(p2)--(p3)--(p4)--(p5); 
\draw (-0.8,0.6)--(2,2);
\draw (p1)--(-0.55,-0.75);
\draw (p5)--(1.75,-0.75);
\draw (0.5,-0.05)--(0.5,0.05);
\draw (1,-0.05)--(1,0.05);
\draw (1.5,-0.05)--(1.5,0.05);
\foreach \t in {1,2,...,5} \fill[black] (p\t) circle (0.05);
\end{tikzpicture}
\caption{$\mu_n(\cF; \sembunij)$ and $\mu(\cF; \sembunij)$.}
\label{fmu}
\end{figure}

Note that $\mu_n$ depends on the orientation of $\sembunij$ but $\mu$ does not, and that $\mu_0(\cF; \sembunij)=\mu_1(\cF; \sembunij)=0$.

\begin{remark}
Let $\mathcal{F}=\bigoplus_{\mathbf{i}\in \mathbb{Z}^2} \alpha_{\mathbf{i}} \mathbf{x}^{\bi}$ be a tropical polynomial, $\sembunij$ a $1$-simplex of $\Delta_{\cF}$ with $\bj-\bi$ primitive, and $\overline{L}$ the corresponding edge of $V(\cF)$ (see Theorem \ref{dualitytheorem}).
Then, for any $P\in \overline{L}$ and $n\in \mathbb{Z}\setminus \{0, 1\}$, we have the following (see Remark \ref{1-dim.cell}):
\[
\alpha_{\mathbf{i}}+\mathbf{i}\cdot P=\alpha_{\mathbf{j}}+\mathbf{j}\cdot P>\alpha_{\mathbf{i}+n(\mathbf{j}-\mathbf{i})}+(\mathbf{i}+n(\mathbf{j}-\mathbf{i}))\cdot P,
\]
and hence,
\[
-\alpha_{\mathbf{i}+n(\mathbf{j}-\mathbf{i})}+\alpha_{\mathbf{i}}+n(\alpha_{\mathbf{j}}-\alpha_{\mathbf{i}})>0,
\]
i.e.,
\[
\mu_n(\cF; \sembunij)>0.
\]
Thus, in this case, we have $\mu(\cF; \sembunij)>0$.
In particular, if $L\in \cSIfg$, $P\in L$ and $\sembunij=\Phi_2(L)$, then $\mu(g; \sembunij)>0$.

The value $\mu(\cF; \sembunij)$ measures the margin for $\sembunij$ to be a $1$-simplex of $\Delta_{\cF}$, in a sense.
\end{remark}

Now, let us state the main theorems.
We consider the following graph theoretic condition which will be crucial in our sufficient conditions.

\begin{definition}
We say that $\cSI'$ is \emph{acyclic} with respect to $\Phi_2$ if the map $\Phi_2|_{\cSI'}$ is injective, i.e., there is no duplication in $\Delta':=\Phi_2(\cSI')$, and the union of the elements of $\Delta'$ is a forest.
\end{definition}

\begin{remark}
The acyclicity of $\cSI'$ is not directly correlated with acyclicity in $\Trop(V(g))$.
Even if $\cSI'$ is acyclic with respect to $\Phi_2$, the union of the corresponding edges of $\Trop(V(g))$ may have cycles (cf. Example \ref{ex3}).
\end{remark}

The following theorem implies that $D$ can be realized on $\cSI' \cup \cPI$ if $\cSI'$ is acyclic with respect to $\Phi_2$ and $D$ is sufficiently close to $E$.

\begin{theorem}(=Theorem \ref{thm_main1})
Let a divisor $D$ satisfy the condition $(*)$ in Definition \ref{star}.
Assume that $\cSI'$ is acyclic with respect to $\Phi_2$ and that for each $L\in \cSI'$, we have $\dist(D|_L, E|_L)<\mu(g; \Phi_2(L))$.
Then, there exists $g'\in k[x^{\pm1}, y^{\pm1}]$ such that $\trop(g')=\trop(g)$ and
\[
\trop(V(f, g'))|_{\cSI' \cup \cPI}=D|_{\cSI' \cup \cPI}.
\]
\end{theorem}

Imposing a further assumption on $\cSI'$, we may drop the restriction on the distance.

\begin{notation}\label{Aff}
For a given set $S\subset \mathbb{R}^n$, we write $\Aff(S)$ for the affine span of $S$.
\end{notation}

\begin{theorem}(=Theorem \ref{thm_main2})
Let a divisor $D$ satisfy the condition $(*)$ in Definition \ref{star}.
Assume that $\cSI'$ is acyclic with respect to $\Phi_2$ and that we can number and order the endpoints of the elements of $\Delta':=\Phi_2(\cSI')$ as $p_1< \dots < p_n$ so that this order is normal on each tree of the forest and that for each element $\sembunpipj$ of $\Delta'$, its affine span $\Aff(\sembunpipj)$ does not contain a point $p_l$ with $l>i, j$.
Then, there exists $g'\in k[x^{\pm1}, y^{\pm1}]$ such that $\trop(g')=\trop(g)$ and
\[
\trop(V(f, g'))|_{\cSI' \cup \cPI}=D|_{\cSI' \cup \cPI}.
\]
\end{theorem}

The proofs of the theorems proceed as follows.
For an element $L\in \cSIfg$, we will give an algorithm to determine $\trop(V(f, g))\cap L$ (see Lemma \ref{lem_f_L}, Definition \ref{H_Elim} and Proposition \ref{prop_d_00}).
This algorithm proceeds by constructing a suitable Laurent polynomial in the ideal $(f, g)$ and tells us how to modify $g$ in order to realize $D$ on $L$.
Using this, we will determine the coefficients of $g'$ one by one.
In the setting of Theorem \ref{thm_main2}, we use the given ordering.
We need the acyclicity condition to maintain the consistency.

\begin{remark}
Let $L\in \cSI'$ and $\Phi_2(L)=\sembunpipj$.
In determining $\trop(V(f, g))\cap L$ and the coefficient $d'_{p_{i}}$ of $g'$, the coefficients $d_{\mathbf{i}}$ for $\mathbf{i}\in \Aff(\sembunpipj)$ are essential.
This is why Theorem \ref{thm_main1} (resp. \ref{thm_main2}) requires the condition about the coefficients $d_{\mathbf{i}}$ for $\mathbf{i}\in \Aff(\sembunpipj)$ (resp. about $\Aff(\sembunpipj)$).
\end{remark}

\begin{remark}
The condition $p_l\notin \Aff(\sembunpipj)$ $(l>i, j)$ in Theorem \ref{thm_main2} depends on the ordering, not just on $\cSI'$.
For example, the order on the left in Figure \ref{fex} satisfies the condition, but the one on the right does not.
\end{remark}

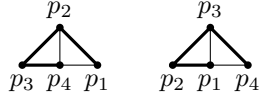
\begin{figure}[H]
\centering
\begin{tikzpicture}
\coordinate [label=below:\text{$p_3$}] (L1) at (0,0);
\coordinate [label=below:\text{$p_4$}] (L2) at (0.5,0);
\coordinate [label=below:\text{$p_1$}] (L3) at (1,0);
\coordinate [label=above:\text{$p_2$}] (L4) at (0.5,0.5);
\coordinate [label=below:\text{$p_2$}] (L5) at (2,0);
\coordinate [label=below:\text{$p_1$}] (L6) at (2.5,0);
\coordinate [label=below:\text{$p_4$}] (L7) at (3,0);
\coordinate [label=above:\text{$p_3$}] (L8) at (2.5,0.5);

\draw (L1)--(L3)--(L4)--cycle;
\draw (L2)--(L4);
\draw (L5)--(L7)--(L8)--cycle;
\draw (L6)--(L8);
\draw [very thick] (L2)--(L1)--(L4)--(L3);
\draw [very thick] (L6)--(L5)--(L8)--(L7);

 \foreach \t in {1,2,...,8} \fill[black] (L\t) circle (0.05);

\end{tikzpicture}
\caption{Two orderings of the endpoints of the elements of $\Delta'=\Phi_2(\cSI')$.}
\label{fex}
\end{figure}

The rest of this paper is organized as follows.
Section 2 gives fundamental definitions and facts about tropical curves.
In Section 3, we show several lemmas concerning properties of $V(\cF)$ and $V(\cG)$ in a neighborhood of $L\in \cSIFG$ and introduce a kind of division procedure for Laurent polynomials over a valuation field.
In Section 4, we explain how to determine $\trop(V(f, g))\cap L$ for $L\in \cSIfg$, and prove the main theorems.
In the last section, we give several examples concerning the main theorems to illustrate the necessity of the acyclicity condition.

\section*{Acknowledgements}
I am grateful to Nobuyoshi Takahashi for helpful comments.
This work was supported by JST, the establishment of university fellowships towards the creation of science technology innovation, Grant Number JPMJFS2129.

\section{Tropical curves}

In this section, we recall the basics about tropical plane curves.
For details, see \cite{MS}.
First, we give the definition of the tropical algebra which is essential for studying tropical geometry.
\begin{definition}[Tropical algebra]
We define $\mathbb{T}=\mathbb{R}\cup \{-\infty\}$.
The \textit{tropical algebra} is the triple $(\mathbb{T}, \oplus, \odot)$, where the addition $\oplus$ is defined as the operation that takes the maximum of two numbers and the multiplication $\odot$ is defined as the ordinary addition.
We can easily check that $(\mathbb{T}, \oplus, \odot)$ is a semifield.
\end{definition}

To define tropical plane curves in terms of tropical algebra, we define tropical polynomials.
\begin{definition}[Tropical polynomials]
A \textit{tropical polynomial} $\mathcal{F}$ is an expression of the form
\[
\mathcal{F}=\bigoplus_{\mathbf{i}} \alpha_{\mathbf{i}} x_1^{i_1}\dots x_n^{i_n},
\]
where $\mathbf{i}=(i_1, \dots, i_n)\in \mathbb{Z}^n$ and $\alpha_{\mathbf{i}}\in \mathbb{T}$, and only finitely many of the coefficients $\alpha_{\mathbf{i}}$ are not $-\infty$.
We may drop terms with coefficients $-\infty$.
A tropical polynomial defines a map from $\mathbb{R}^n$ to $\mathbb{R}\cup \{-\infty\}$ in a natural way:
\[
\mathcal{F}(t_1, \dots, t_n)=\max_{\mathbf{i}}(\alpha_{\mathbf{i}}+i_1t_1+\dots+i_nt_n).
\]
We write $\mathbb{T}[x_1^{\pm1}, \dots, x_n^{\pm1}]$ for the set of all $n$-variate tropical polynomials, and define the addition and the multiplication in a natural way.
\end{definition}

\begin{definition}[Tropical hypersurfaces]
Let $\mathcal{F}=\bigoplus_{\mathbf{i}} \alpha_{\mathbf{i}} x_1^{i_1}\dots x_n^{i_n}\neq -\infty$ be a tropical polynomial.
The \textit{tropical hypersurface} $V(\mathcal{F})$ defined by $\mathcal{F}$ is the set
\begin{eqnarray*} 
V(\mathcal{F})=\left\{ (t_1,\dots, t_n)\in \mathbb{R}^n \middle| 
\begin{array}{l}
\text{$\exists \mathbf{i}=(i_1, \dots, i_n), \mathbf{j}=(j_1, \dots, j_n) \in \mathbb{Z}^n$ ($\mathbf{i} \neq \mathbf{j}$) s.t.}\\
\text{$\alpha_{\mathbf{i}}+i_1t_1+\dots+i_nt_n=\alpha_{\mathbf{j}}+j_1t_1+\dots+j_nt_n$}\\
\hspace{33.5mm} =\mathcal{F}(t_1,\dots, t_n)
\end{array}
\right\}.
\end{eqnarray*}
If $\mathcal{F}=-\infty$, i.e. all the coefficients of $\mathcal{F}$ are $-\infty$, we define $V(-\infty)=\mathbb{R}^n$.
When $n=2$ and $\mathcal{F}\neq -\infty$, we call $V(\mathcal{F})$ a \textit{tropical plane curve}.
Later, we will consider a tropical plane curve as a polyhedral complex endowed with weights on its $1$-dimensional cells (see Definition \ref{weight}).
\end{definition}

The following map is a bridge between algebraic geometry and tropical geometry.
\begin{definition}[Tropicalization map]
We define the \textit{tropicalization map} as follows:
\begin{eqnarray*}
\trop: &(k^*)^n&\to{\mathbb{R}}^n\\
&(x_1, \dots, x_n)&\mapsto (-\val(x_1), \dots, -\val(x_n)).
\end{eqnarray*}
\end{definition}

\begin{definition}[Tropicalization of Laurent polynomials]\label{TLP}
Let $f=\sum_{\mathbf{i}} c_{\mathbf{i}} \mathbf{x}^{\bi}\in k[x_1^{\pm1},\dots, x_n^{\pm1}]$ be a Laurent polynomial.
We define the \textit{tropicalization} of $f$ as
\[
\trop(f)=\bigoplus_{\mathbf{i}} \trop(c_{\mathbf{i}}) \mathbf{x}^{\bi}\ \left(=\bigoplus_{\mathbf{i}} \left((-\val(c_{\mathbf{i}})) \mathbf{x}^{\bi}\right)\right).
\]
\end{definition}

\begin{notation}
For $A\subset (k^*)^n$, we write $\Trop(A)$ for the closure of $\trop(A)$ in $\mathbb{R}^n$.
\end{notation}

Recall that $k$ is an algebraically closed field with a nontrivial valuation.

\begin{theorem}[Kapranov's Theorem, {\cite[Theorem 2.1.1]{EKL}}]\label{MS3.1.3}
Let $f\in k[x_1^{\pm1},\dots, x_n^{\pm1}]$ be a Laurent polynomial.
Then, we have
\[
V(\trop(f))=\Trop(V(f)).
\]
\end{theorem}

\begin{definition}[Dual subdivisions, {\cite[Definition 3.10]{Mik}}]\label{dualsubdiv}
Let $\mathcal{F}=\bigoplus_{i,j} \alpha_{ij} x^i y^j$ be a tropical polynomial.
We write $\Newt(\mathcal{F})\subset \mathbb{R}^2$ for the convex hull of the set $\{(i, j)\in \mathbb{Z}^2\ |\ \alpha_{ij}\neq -\infty\}$.
Let $A_{\mathcal{F}}\subset \mathbb{R}^3$ be the convex hull of the set
\[
\{ (i, j, \alpha)\in \mathbb{Z}^2\times \mathbb{R}\ \mid \alpha \leq \alpha_{ij}\}.
\]
Then, the projections of the bounded faces of $A_{\mathcal{F}}$ form a lattice subdivision of $\Newt(\mathcal{F})$.
This naturally has a structure of a polyhedral complex.
The \textit{dual subdivision} of $\mathcal{F}$ is this polyhedral complex and we denote it by $\Delta_{\mathcal{F}}$.
For a Laurent polynomial $f\in k[x^{\pm1}, y^{\pm1}]$, we also write $\Delta_f$ for the dual subdivision of $\trop(f)$.
\end{definition}

\begin{theorem}[The Duality Theorem, {\cite[Proposition 3.11]{Mik}}]\label{dualitytheorem}
Let $\Gamma=V(\mathcal{F})$ be a tropical plane curve.
Then, $\Gamma$ is the support of a finite $1$-dimensional polyhedral complex $\Sigma_{\mathcal{F}}$ (possibly with noncompact cells) in $\mathbb{R}^2$.
It is dual to the subdivision $\Delta_{\mathcal{F}}$ in the following sense:
\begin{itemize}
  \item (Closures of) domains of $\mathbb{R}^2\setminus \Gamma$ correspond to lattice points in $\Delta_\mathcal{F}$.
  \item $1$-dimensional cells in $\Sigma_{\mathcal{F}}$ correspond to $1$-simplexes in $\Delta_{\mathcal{F}}$.
   \item $0$-dimensional cells in $\Sigma_{\mathcal{F}}$ correspond to $2$-dimensional cells in $\Delta_{\mathcal{F}}$.
  \item This correspondence is inclusion-reversing.
  \item A $1$-dimensional cell in $\Sigma_{\mathcal{F}}$ is orthogonal to the corresponding $1$-simplex in $\Delta_{\mathcal{F}}$ (see Figure \ref{fdual}).
\end{itemize}
For a cell $\sigma \in \Delta_{\mathcal{F}}$, the corresponding cell in $\Sigma_{\cF}$ is given by $\{P\in \mathbb{R}^2\ |\ \cF(P)=\alpha_{\bi}+\bi\cdot P \text{ for any vertex $\bi$ of $\sigma$}\}$.
In particular, $1$-dimensional cells in $\Sigma_{\mathcal{F}}$ have rational slopes. 
\end{theorem}

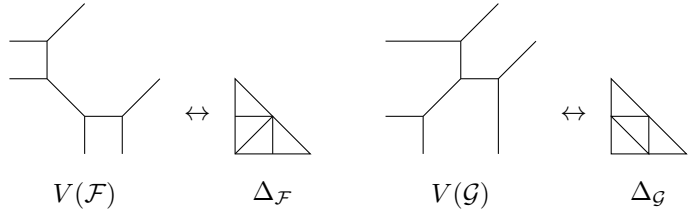
\begin{figure}[H]
\centering
\begin{tikzpicture}
\coordinate (L1) at (2,0);
\coordinate (L2) at (2.5,0);
\coordinate (L3) at (3,0);
\coordinate (L4) at (2,0.5);
\coordinate (L5) at (2.5,0.5);
\coordinate (L6) at (2,1);

\coordinate (L7) at (-1,1);
\coordinate (L8) at (-0.5,1);
\coordinate (L9) at (0,0.5);
\coordinate (L10) at (0.5,0.5);
\coordinate (L11) at (1,1);
\coordinate (L12) at (0,0);
\coordinate (L13) at (0.5,0);
\coordinate (L14) at (-1,1.5);
\coordinate (L15) at (-0.5,1.5);
\coordinate (L16) at (0,2);

\draw (L1)--(L3)--(L6)--cycle;
\draw (L4)--(L5)--(L2);
\draw (L1)--(L5);

\draw (L11)--(L10)--(L9)--(L8)--(L15)--(L16);
\draw (L14)--(L15);
\draw (L7)--(L8);
\draw (L10)--(L13);
\draw (L9)--(L12);

\coordinate (A1) at (7,0);
\coordinate (A2) at (7.5,0);
\coordinate (A3) at (8,0);
\coordinate (A4) at (7,0.5);
\coordinate (A5) at (7.5,0.5);
\coordinate (A6) at (7,1);

\coordinate (A7) at (4,0.5);
\coordinate (A8) at (4.5,0.5);
\coordinate (A9) at (5,1);
\coordinate (A10) at (5.5,1);
\coordinate (A11) at (6,1.5);
\coordinate (A12) at (4.5,0);
\coordinate (A13) at (5.5,0);
\coordinate (A14) at (4,1.5);
\coordinate (A15) at (5,1.5);
\coordinate (A16) at (5.5,2);

\draw (A1)--(A3)--(A6)--cycle;
\draw (A4)--(A5)--(A2)--cycle;

\draw (A7)--(A8)--(A9)--(A10)--(A11);
\draw (A14)--(A15)--(A16);
\draw (A9)--(A15);
\draw (A8)--(A12);
\draw (A10)--(A13);

\coordinate [label=below:\text{$\Delta_{\mathcal{F}}$}] (a) at (2.5,-0.25);
\coordinate [label=below:\text{$\Delta_{\mathcal{G}}$}] (b) at (7.5,-0.25);
\coordinate [label=below:\text{$V(\mathcal{F})$}] (a) at (0,-0.25);
\coordinate [label=below:\text{$V(\mathcal{G})$}] (b) at (5,-0.25);

\coordinate [label=below:\text{$\leftrightarrow$}] (a) at (1.5,0.7);
\coordinate [label=below:\text{$\leftrightarrow$}] (b) at (6.5,0.7);

\end{tikzpicture}
\caption{(Smooth) tropical plane curves and their dual subdivisions.}
\label{fdual}
\end{figure}

\begin{notation}
Let $\Gamma$ be a tropical plane curve.
We call a $0$-dimensional cell of $\Gamma$ a vertex of $\Gamma$ and a $1$-dimensional cell of $\Gamma$ an edge of $\Gamma$.
\end{notation}

We define the weight of an edge of a tropical plane curve using the dual subdivision.

\begin{definition}\label{weight}
Let $\Gamma=V(\mathcal{F})$ be a tropical plane curve and $\sigma \in \Sone(\Gamma)$ an edge of $\Gamma$.
The \textit{weight} $w_{\sigma}$ of $\sigma$ in $\Gamma$ is the lattice length of the corresponding $1$-simplex of $\Delta_{\mathcal{F}}$.
\end{definition}

From now on, a ``tropical curve'' will refer to the polyhedral set $\Gamma$ together with weights on its edges.

\begin{remark}\label{1-dim.cell}
Let $\mathcal{F}=\bigoplus_{i,j} \alpha_{ij} x^i y^j$ be a tropical polynomial.
Let $\sigma$ be an edge of $\Sigma_{\mathcal{F}}$ with weight $1$ and $\sembunij$ the corresponding $1$-simplex of $\Delta_{\mathcal{F}}$.
Assume that $\alpha_{\mathbf{k}}+\mathbf{k}\cdot P=\mathcal{F}(P)$ for $P\in \sigma$ and $\mathbf{k}\in \mathbb{Z}^2\setminus \{\bi, \bj\}$.
Then $\mathbf{k}$ is one of the vertices of a $2$-dimensional cell in $\Delta_{\cF}$, corresponding to a vertex of $\sigma$, containing $\sembunij$ as its face.
In particular, for any $\mathbf{k}\in (\Aff(\sembunij)\cap \mathbb{Z}^2)\setminus \{\bi, \bj\}$, we have
\[
\alpha_{\mathbf{k}}+\mathbf{k}\cdot P<\mathcal{F}(P).
\]
\end{remark}

\begin{notation}\label{v_P,L}
Let $\Gamma$ be a tropical plane curve, $P$ a vertex of $\Gamma$ and $L$ an edge of $\Gamma$ containing $P$.
Let $R$ be the ray which contains $L$ such that $P$ is its endpoint.
We denote by $\mathbf{v}_{P, L}$ the primitive vector that have the same direction as $R$.
\end{notation}

Tropical plane curves satisfy the following balancing condition.

\begin{theorem}\label{MS3.3.2}
\cite[Theorem 3.3.2]{MS}
Let $\Gamma$ be a tropical plane curve and $P$ a vertex of $\Gamma$ and $L_1, \dots, L_n$ the edges of $\Gamma$ containing $P$ with weights $w_{L_i}$.
Then, we have
\[
\sum_{i}w_{L_i}\mathbf{v}_{P, L_i}=\mathbf{0}.
\]
\end{theorem}

\begin{definition}
Let $\Gamma=V(\mathcal{F})$ be a tropical plane curve.
A vertex $P\in \Gamma$ is called \textit{smooth} if the area of the corresponding cell in $\Delta_{\mathcal{F}}$ is $1/2$.
We see that this is equivalent to the condition that it is trivalent and all the weights of the three edges $L_1$, $L_2$ and $L_3$ containing $P$ are $1$, and for some (or any) pair $(i, j)$ ($i, j\in \{1, 2, 3\}$, $i\neq j$), we have $|\det(\mathbf{v}_{P, L_i}, \mathbf{v}_{P, L_j})|=1$.
\end{definition}

\begin{definition}[Smooth tropical plane curves]\label{STC}
A tropical plane curve $\Gamma=V(\mathcal{F})$ is called \textit{smooth} if all the lattice lengths of the $1$-simplexes of $\Delta_{\mathcal{F}}$ are $1$ and all the areas of the $2$-dimensional cells of $\Delta_{\mathcal{F}}$ are $1/2$ (see Figure \ref{fdual}).
In other words, $\Gamma$ is smooth if all the vertices are smooth and all the weights of the edges are $1$.
\end{definition}

\begin{notation}
Let $\Gamma$ be a tropical plane curve and $\sigma$ an edge of $\Gamma$.
We denote by $\mathbf{v}_{\sigma}$ primitive vector that have the same direction as $\Aff(\sigma)$.
This is well-defined up to sign.
\end{notation}

\begin{definition}[Transverse intersection points]
Let $\Gamma_1$ and $\Gamma_2$ be tropical plane curves.
A point $P$ is a \textit{transverse intersection point} of $\Gamma_1$ and $\Gamma_2$ if it is a proper intersection point of them and is a vertex of neither of them.
For a transverse intersection point $P$, there exist unique edges $L_i\in \Sone(\Gamma_i)$ ($i=1, 2$) containing $P$ in their interiors.
In this case, we say that $L_1$ and $L_2$ intersect \textit{transversely} at $P$, and we define the \textit{intersection multiplicity} at $P$ as
\[
i(P; \Gamma_1 \cdot \Gamma_2):=w_{L_1}w_{L_2}|\det(\mathbf{v}_{L_1}, \mathbf{v}_{L_2})|.
\]
Tropical plane curves $\Gamma_1$ and $\Gamma_2$ intersect \textit{transversely} if all the points in $\Gamma_1\cap \Gamma_2$ are transverse intersection points.
\end{definition}


Note that for a tropical plane curve $\Gamma$ and a vector $\mathbf{v}\in \mathbb{R}^2$, $\Gamma+\mathbf{v}$ is a tropical plane curve.
For tropical plane curves $\Gamma_1$ and $\Gamma_2$, it is known that for a generic vector $\mathbf{v}\in \mathbb{R}^2$ and a nonzero real number $\epsilon \in \mathbb{R}$ with sufficiently small absolute value, $\Gamma_1+\epsilon \mathbf{v}$ and $\Gamma_2$ intersect transversely, and the following sum is well-defined (see \cite[Section 6]{OR}).

\begin{definition}[Intersection multiplicities]\label{IM}
Let $\Gamma_1$ and $\Gamma_2$ be tropical plane curves.
We define the \textit{intersection multiplicity} at a point $P\in \Gamma_1 \cap \Gamma_2$ as
\[
i(P; \Gamma_1 \cdot \Gamma_2):=\sum_{L_1 \ni P, L_2\ni P}\left(\sum_{Q\in (L_1+\epsilon \mathbf{v})\cap L_2}i(Q; (\Gamma_1+\epsilon \mathbf{v}) \cdot \Gamma_2)\right),
\]
where $L_1$ and $L_2$ are edges of $\Gamma_1$ and $\Gamma_2$, $\mathbf{v}\in \mathbb{R}^2$ is a generic vector, $\epsilon \in \mathbb{R}$ is a sufficiently small nonzero real number.
\end{definition}

It is easy to see that $i(P; \Gamma_1 \cdot \Gamma_2)=i(P; \Gamma_2 \cdot \Gamma_1)$.

\begin{definition}[Stable intersection divisor]\label{SID}
Let $\Gamma_1$ and $\Gamma_2$ be tropical plane curves.
The \textit{stable intersection divisor} of $\Gamma_1$ and $\Gamma_2$ is defined as
\[
\sum_{P\in \Gamma_1\cap \Gamma_2}i(P; \Gamma_1 \cdot \Gamma_2)P.
\]
In an appropriate sense, this is equal to the limit of $(\Gamma_1+\epsilon \mathbf{v})\cap \Gamma_2$ as $\epsilon \rightarrow 0$, where $\mathbf{v}\in \mathbb{R}^2$ is a generic vector and $\epsilon$ is a sufficiently small nonzero real number.
\end{definition}





The following theorem says that the tropicalization conserves the intersection number in a certain sense.

\begin{theorem}\label{OR6.13}
\cite[Corollary 6.13]{OR}
Let $X_1, \dots, X_m\in (k^*)^n$ be pure dimensional closed subschemes of $(k^*)^n$ with $\sum_{i}\codim (X_i)=n$.
Let $\mathfrak{K}$ be a connected component of $\bigcap_{i}\Trop(X_i)$, and suppose that $\mathfrak{K}$ is bounded.
Then there are only finitely many $k$-valued points $x\in \left(\bigcap_{i}X_i\right)(k)$ with $\trop(x)\in \mathfrak{K}$, and
\[
\sum_{\substack{x\in \left(\bigcap_{i}X_i\right)(k)\\\trop(x)\in \mathfrak{K}}}i(x; X_1\cdots X_m)=\sum_{P\in \mathfrak{K}}i(P; \Trop(X_1)\cdots \Trop(X_m)).
\]
\end{theorem}

\section{Preparations for the main theorems}

We will provide additional explanations for some facts from Section 1 and make preparations for the next section.
Since we are going to compare the valuations of different terms in polynomials, we make the following definition.
\begin{definition}
We define a map $\tau$ as follows:
\begin{eqnarray*}
\tau: \mathbb{T}[x^{\pm1}, y^{\pm1}]\times \mathbb{Z}^2\times \mathbb{R}^2&\to&\mathbb{R}\cup \{-\infty\}\\
(\bigoplus_{\mathbf{i}} \alpha_{\mathbf{i}} \mathbf{x}^{\mathbf{i}}; \mathbf{j}; P)&\mapsto&\alpha_{\mathbf{j}}+\mathbf{j}\cdot P.
\end{eqnarray*}
For a Laurent polynomial $f\in k[x^{\pm1}, y^{\pm1}]$, we write $\tau(f; \mathbf{j}; P)$ for $\tau(\trop(f); \mathbf{j}; P)$.
\end{definition}


The map $\tau$ satisfies the following.
\begin{lemma}\label{tau(f+g)}
Let $f, g\in k[x^{\pm1}, y^{\pm1}]$ be Laurent polynomials.
Then, for all $\mathbf{i}\in \mathbb{Z}^2$ and $P\in \mathbb{R}^2$, we have
\[
\tau(f+g; \mathbf{i}; P)\leq \max \{ \tau(f; \mathbf{i}; P), \tau(g; \mathbf{i}; P)\}.
\]
Moreover, the equality holds if $\tau(f; \mathbf{i}; P)\neq \tau(g; \mathbf{i}; P)$.
\end{lemma}

\begin{proof}
This is clear from the ultrametric inequality for the valuation.
\end{proof}

For $A=(a_{ij})\in \mathrm{GL}_2(\mathbb{Z})$, $\mathbf{b}=(b_1, b_2)\in \mathbb{R}^2$ and $\mathbf{t}=(t_1, t_2)\in (k^*)^2$ such that $\trop(\mathbf{t})=\mathbf{b}$, we define the following automorphisms and an affine transformation.
\begin{eqnarray*}
\phi: (k^*)^2&\to&(k^*)^2\\
(a, b)&\mapsto&(a^{a_{11}}b^{a_{12}}t_1, a^{a_{21}}b^{a_{22}}t_2),\\
\phi^*: k[x^{\pm1}, y^{\pm1}]&\to& k[x^{\pm1}, y^{\pm1}]\\
x&\mapsto&x^{a_{11}}y^{a_{12}}t_1,\ \ y\mapsto x^{a_{21}}y^{a_{22}}t_2,\\
\trop(\phi)=\Phi: \mathbb{R}^2&\to&\mathbb{R}^2\\
\mathbf{v}&\mapsto&A\mathbf{v}+\mathbf{b},\\
\Phi^*: \mathbb{T}[x^{\pm1}, y^{\pm1}]&\to& \mathbb{T}[x^{\pm1}, y^{\pm1}]\\
x&\mapsto&b_1x^{a_{11}}y^{a_{12}},\ \ y\mapsto b_2x^{a_{21}}y^{a_{22}},\\
{{}^{t}\Phi}^{-}: \mathbb{R}^2&\to&\mathbb{R}^2\\
\mathbf{v}&\mapsto&{{}^{t}\!A}^{-1}(\mathbf{v}-\mathbf{b}).
\end{eqnarray*}
Then, the following can be verified by direct calculations.
\begin{itemize}
\item $\forall P\in (k^*)^2,\ \trop(\phi(P))=\Phi(\trop(P))$.
\vspace{1mm}
\item $\forall f\in k[x^{\pm1}, y^{\pm1}],\ \Trop(V(\phi^*(f)))=V(\Phi^*(\trop(f)))$.
\vspace{1mm}
\item $\forall f\in k[x^{\pm1}, y^{\pm1}],\ \forall P\in (k^*)^2,\ f(\phi(P))=(\phi^*(f))(P)$.
\vspace{1mm}
\item $\forall \cF \in \mathbb{T}[x^{\pm1}, y^{\pm1}],\ {{}^{t}\Phi}^{-}(V(\cF))=V(\Phi^{*}(\cF))$.
\end{itemize}
\medbreak
Thus, if $L$ is an edge of $\Trop(V(f))$, we can find an automorphism $\phi$ of $(k^*)^2$ such that $\trop(\phi)(L)$ is contained in the $y$-axis, for example.

Recall that $\cRayFG$ and $\cLSFG$ were the sets of rays and line segments contained in $V(\cF)\cap V(\cG)$, defined in Definition \ref{RLS}, and that $\cSIFG=\cRayFG \cup \cLSFG$.

\begin{lemma}\label{L}
Let $\cF$ and $\cG$ be bivariate tropical polynomials.
\begin{enumerate}
\item Let $L\in \cLSFG$, and $P_+$ and $P_-$ the endpoints of $L$.
Then, $P_*$ (``$*=+$ or $-$'') is a smooth vertex in one of $V(\cF)$ and $V(\cG)$, and is in the interior of an edge of weight $1$ in the other.
Furthermore, the interior of $L$ contains no vertices of $V(\cF)$ and $V(\cG)$.
In other words, for a neighborhood $U$ of $L$, the restrictions of $V(\cF)$ and $V(\cG)$ to $U$ are as in Figure \ref{fL}, where each vertex is smooth and each edge has weight $1$.

In particular, the stable intersection points of $V(\cF)$ and $V(\cG)$ on $L$ are the endpoints of $L$, each with weight $1$.

\item For $L\in \cRayFG$, the endpoint of $L$ is a smooth vertex in one of $V(\cF)$ and $V(\cG)$, and is in the interior of an edge of weight $1$ in the other.
Furthermore, the interior of $L$ contains no vertices of $V(\cF)$ and $V(\cG)$.
\end{enumerate}
\end{lemma}

\begin{figure}[H]
\centering
\begin{tikzpicture}
\coordinate [label=left:$P_-$] (A1) at (0.4,0.4);
\coordinate [label=left:$P_+$] (A2) at (0.4,1.6);
\coordinate [label=right:$P_-$] (A3) at (1.6,0.4);
\coordinate [label=right:$P_+$] (A4) at (1.6,1.6);
\coordinate [label=left:$P_-$] (A5) at (3.4,0.4);
\coordinate [label=left:$P_+$] (A6) at (3.4,1.6);
\coordinate [label=right:$P_-$] (A7) at (4.6,0.4);
\coordinate [label=right:$P_+$] (A8) at (4.6,1.6);
\coordinate [label=left:$P_-$] (A9) at (6.4,0.4);
\coordinate [label=left:$P_+$] (A10) at (6.4,1.6);
\coordinate [label=right:$P_-$] (A11) at (7.6,0.4);
\coordinate [label=right:$P_+$] (A12) at (7.6,1.6);
\coordinate [label=left:$P_-$] (A13) at (9.4,0.4);
\coordinate [label=left:$P_+$] (A14) at (9.4,1.6);
\coordinate [label=right:$P_-$] (A15) at (10.6,0.4);
\coordinate [label=right:$P_+$] (A16) at (10.6,1.6);

\draw (0,0)--(A1)--(0.8,0);
\draw (0,2)--(A2)--(0.8,2);
\draw (A1)--(A2);
\draw (1.6,0)--(1.6,2);
\draw[dotted] (0,0)--(0.8,0)--(0.8,2)--(0,2)--cycle;
\draw[dotted] (1.2,0)--(2,0)--(2,2)--(1.2,2)--cycle;

\draw (4.2,0)--(A7)--(5,0);
\draw (4.2,2)--(A8)--(5,2);
\draw (A7)--(A8);
\draw (3.4,0)--(3.4,2);
\draw[dotted] (3,0)--(3.8,0)--(3.8,2)--(3,2)--cycle;
\draw[dotted] (4.2,0)--(5,0)--(5,2)--(4.2,2)--cycle;

\draw (7.2,0)--(A11)--(8,0);
\draw (6,2)--(A10)--(6.8,2);
\draw (A10)--(6.4,0);
\draw (A11)--(7.6,2);
\draw[dotted] (6,0)--(6.8,0)--(6.8,2)--(6,2)--cycle;
\draw[dotted] (7.2,0)--(8,0)--(8,2)--(7.2,2)--cycle;

\draw (9,0)--(A13)--(9.8,0);
\draw (10.2,2)--(A16)--(11,2);
\draw (A16)--(10.6,0);
\draw (A13)--(9.4,2);
\draw[dotted] (9,0)--(9.8,0)--(9.8,2)--(9,2)--cycle;
\draw[dotted] (10.2,0)--(11,0)--(11,2)--(10.2,2)--cycle;

\foreach \t in {1,...,16} \fill[black] (A\t) circle (0.045);

\coordinate [label=below:(a)] (a) at (1,-0.1);
\coordinate [label=below:(b)] (b) at (4,-0.1);
\coordinate [label=below:(c)] (c) at (7,-0.1);
\coordinate [label=below:(d)] (d) at (10,-0.1);
\end{tikzpicture}
\caption{$V(\cF)$ and $V(\cG)$ in a neighborhood of $L\in \cLSFG$.}
\label{fL}
\end{figure}

\begin{proof}
First, note that each endpoint of $L$ is a vertex of at least one of $V(\cF)$ and $V(\cG)$, and that if $P\in L$ is a vertex of $V(\cF)$ or $V(\cG)$, then we have $i(P; V(\cF)\cdot V(\cG))\geq 1$.
It follows that $V(\cF)$ and $V(\cG)$ intersect with multiplicity $1$ at each endpoint of $L$, and that $V(\cF)$ and $V(\cG)$ do not have a vertex in the interior of $L$.
Hence, an edge of $V(\cF)$ (resp. $V(\cG)$) intersecting the interior of $L$ contains $L$.

Let us prove (1).
The proof of (2) is similar.
We can assume that $L$ is contained in the $y$-axis and $P_+=(0, a_1)$ and $P_-=(0, a_2)$ with $a_1>a_2$ by applying an affine transformation with a unimodular integral coefficient matrix.
Let $U_+$ be a sufficiently small neighborhood of $P_+$ and $D_R:=\{(p_1, p_2)\in \mathbb{R}^2\ |\ p_1>0\}$.
Let $L_1, \dots, L_s$ be the edges of $V(\cF)$ which intersect $U_+$, and $L'_1, \dots, L'_t$ the edges of $V(\cG)$ which intersect $U_+$, with $L_1\supset L$ and $L'_1\supset L$.
Assume that $P_+$ is a vertex in both of $V(\cF)$ and $V(\cG)$, i.e. $s\geq 3$ and $t\geq 3$.
Then, by the balancing condition at $P_+$, there are $i$ and $j$ such that
\[
D_R\cap L_i\neq \emptyset \text{ and }D_R\cap L'_j\neq \emptyset.
\]
Note that by the assumption that $L\in \cLS(\cF, \cG)$ is a connected component of $V(\cF) \cap V(\cG)$, we have $L_i\cap L'_j=\{P_+\}$.
By symmetry, we assume that the slope of $L_i$ is larger than that of $L'_j$.
Let $\mathbf{v}=(v_1, v_2)\in \mathbb{R}^2$ be a general vector such that $v_1, v_2>0$ and $v_2/v_1$ is sufficiently large, and $\epsilon>0$ a sufficiently small positive number.
We will consider $V(\cF)\cap (\epsilon \mathbf{v}+V(\cG))$.
Then, we have $L_i\cap (\epsilon \mathbf{v}+L'_j)\neq \emptyset$ and $L_i\cap (\epsilon \mathbf{v}+L'_1)\neq \emptyset$.
Hence, we have $i(P_+; V(\cF) \cdot V(\cG))\geq 2$, contradicting to what we saw at the beginning.
Thus $P_+$ is a vertex of exactly one of $V(\cF)$ and $V(\cG)$.

Assume that $P_+$ is a vertex of $V(\cF)$.
Then, by the definition of $\cLSFG$ (see Definition \ref{RLS}), the multiplicity of $L_1$ is $1$.
Since the intersection multiplicity at $P_+$ is $1$, it is cleat that exactly one of $\{L_2, \dots, L_s\}$, say $L_2$, intersects $D_R$, that the weights of $L'_1$ and $L_2$ are both $1$ and that $|\det(\mathbf{v}_{L_2}, \mathbf{v}_{L'_1})|=1$.
Similarly, there is a unique edge of $V(\cF)$ intersecting $U_+\cap \{(p_1, p_2)\in \mathbb{R}^2\ |\ p_1<0\}$, and $V(\cF)$ is trivalent at $P_+$ (note that $V(\cF)$ contains no edge intersecting $U_+\cap \{(p_1, p_2)\in \mathbb{R}^2\ |\ p_1=0,\ p_2>a_1\}$ since $P_+$ is an endpoint of $L$).
By the balancing condition, $P_+$ is a smooth vertex of $V(\cF)$.
The same holds at the point $P_-$.
\end{proof}

\begin{corollary}\label{il-io}
Let $\cF, \cG \in \mathbb{T}[x^{\pm1}, y^{\pm1}]$, $L\in \cSIFG$, $\Phi_1(L)=\sembunioil$ and $\Phi_2(L)=\sembunjojl$.
Then $\bil-\bio=\pm(\bjl-\bjo)$.
\end{corollary}

\begin{proof}
It is clear that $\Aff(\Phi_1(L))=\Aff(\Phi_2(L))$, and hence, it is sufficient to show that $\Phi_1(L)$ and $\Phi_2(L)$ have the same lattice length.
By Lemma \ref{L}, each edge of $V(\cF)$ and $V(\cG)$ intersecting $L$ has weight $1$.
Then, by the definition of the weight of an edge of a tropical plane curve (see Definition \ref{weight}), the lattice lengths of $\Phi_1(L)$ and $\Phi_2(L)$ are $1$.
\end{proof}

\begin{lemma}\label{Lnohashi}
Let $\cF=\bigoplus_{\mathbf{i}}\alpha_{\mathbf{i}}\mathbf{x}^{\mathbf{i}}, \cG=\bigoplus_{\mathbf{i}}\beta_{\mathbf{i}}\mathbf{x}^{\mathbf{i}}\in \mathbb{T}[x^{\pm1}, y^{\pm1}]$ be tropical polynomials, $L\in \cSI(\cF, \cG)$, and $P_+$ an endpoint of $L$.
Assume that $\Phi_1(L)=\Phi_2(L)=\sembunioil$, $\alpha_{\bio}=\beta_{\bio}$ and $\alpha_{\bil}=\beta_{\bil}$.
Let $U_+$ be a sufficiently small neighborhood of $P_+$ and $\overline{L}=\Aff(L)$ (see Notation \ref{Aff}).
By Lemma \ref{L}, for either $(\cF_1, \cF_2)=(\cF, \cG)$ or $(\cF_1, \cF_2)=(\cG, \cF)$, the point $P_+$ is a smooth vertex of $V(\cF_1)$ and is in the interior of an edge of multiplicity $1$ in $V(\cF_2)$.
Let $\sigma_+$ be the $2$-simplex of $\Delta_{\cF_1}$ corresponding to $P_+$ and $\bip$ the vertex of $\sigma_+$ other than $\bio$ and $\bil$.
Then, for all $P\in U_+\cap (\overline{L}\setminus L)$, we have
\begin{eqnarray*}
\tau(\cF_1; \bip; P)>\tau(\cF; \mathbf{i}; P), \tau(\cG; \mathbf{i}; P), \tau(\cF_2; \bip; P)\ (\mathbf{i}\in \mathbb{Z}^2\setminus \{\bip\}),
\end{eqnarray*}
and
\begin{eqnarray*}
\tau(\cF_1; \bip; P_+)&=&\tau(\cF; \bi; P_+)=\tau(\cG; \bi; P_+)\ (\text{$\bi=\bio$ or $\bil$})\\
&>&\tau(\cF; \mathbf{j}; P_+), \tau(\cG; \mathbf{j}; P_+), \tau(\cF_2; \bip; P_+)\ (\mathbf{j}\in \mathbb{Z}^2\setminus \{\bio, \bil, \bip\}).
\end{eqnarray*}
\end{lemma}

\begin{proof}
By symmetry, we may assume that $\cF_1=\cF$.
Let $P\in  U_+\cap (\overline{L}\setminus L)$.
Then, the restrictions of the two tropical plane curves to a neighborhood of $P_+$ are as in Figure \ref{f_1_m_+}.
We have
\begin{eqnarray}
\tau(\cF; \bip; P)>\tau(\cF; \mathbf{i}; P)\ \ (\mathbf{i}\in \mathbb{Z}^2\setminus \{\bip\}),
\end{eqnarray}
since $\bip$ is the vertex corresponding to the domain containing $P$, and for all $\mathbf{j}\in \mathbb{Z}^2\setminus \{\bio, \bil, \bip\}$, we have
\begin{eqnarray}
\tau(\cF; \bip; P_+)=\tau(\cF; \bio; P_+)=\tau(\cF; \bil; P_+)>\tau(\cF; \mathbf{j}; P_+).
\end{eqnarray}
For all $\mathbf{i}\in \mathbb{Z}^2\setminus \{\bio, \bil\}$, we have
\begin{eqnarray}
&\tau(\cG; \bio; P)=\tau(\cG; \bil; P)>\tau(\cG; \mathbf{i}; P),&\\
&\tau(\cG; \bio; P_+)=\tau(\cG; \bil; P_+)>\tau(\cG; \mathbf{i}; P_+).&
\end{eqnarray}
By the assumption that $\alpha_{\bio}=\beta_{\bio}$ and $\alpha_{\bil}=\beta_{\bil}$, we have
\begin{eqnarray}
\ \tau(\cF; \bi; P)=\tau(\cG; \bi; P),\ \tau(\cF; \bi; P_+)=\tau(\cG; \bi; P_+) \ \ \text{($\bi=\bio$ or $\bil$)}.
\end{eqnarray}
By the inequalities (1), (3) and (5), we have
\begin{eqnarray*}
\tau(\cF; \bip; P)>\tau(\cF; \bio; P)= \tau(\cG; \bio; P)\geq \tau(\cG; \bi; P)\ \ (\bi \in \mathbb{Z}^2).
\end{eqnarray*}
The second inequalities follow from (2), (4) and (5).
\end{proof}

\begin{figure}[H]
\centering
\begin{tikzpicture}
\draw (0,0) circle [radius=1.2];
\draw (4,0) circle [radius=1.2];
\draw (-0.7,-0.45) circle [radius=0.275];
\draw (0.7,-0.45) circle [radius=0.275];
\draw (0,0.8) circle [radius=0.275];
\draw (3.3,0) circle [radius=0.275];
\draw (4.8,0) circle [radius=0.275];

\coordinate (A1) at (0,0);
\coordinate [label=right:$P_+$] (B1) at (-0.05,-0.1);
\coordinate (A2) at (0,0.3);
\coordinate [label=right:$P$] (D1) at (-0.05,0.325);
\coordinate (A3) at (0,-2);
\coordinate (A4) at (4,0);
\coordinate [label=right:$P_+$] (C1) at (3.95,-0.1);
\coordinate (A5) at (4,0.3);
\coordinate [label=right:$P$] (C2) at (3.95,0.33);
\coordinate (A6) at (4,-2);
\coordinate [label=right:$L$] (A7) at (0,-1.6);
\coordinate [label=right:$L$] (A8) at (4,-1.6);
\coordinate [label=above:$\overline{L}$] (B2) at (0,1.4);
\coordinate [label=above:$\overline{L}$] (B3) at (4,1.4);
\coordinate [label=below:$U_+$] (B4) at (-1,1.4);
\coordinate [label=below:$U_+$] (B5) at (3,1.4);

\coordinate [label=above:$\bip$] (a1) at (0.01,0.53);
\coordinate [label=above:$\bio$] (a2) at (-0.7,-0.69);
\coordinate [label=above:$\bil$] (a3) at (0.7,-0.69);
\coordinate [label=above:$\bio$] (a4) at (3.3,-0.24);
\coordinate [label=above:$\bil$] (a5) at (4.8,-0.24);

\draw (A1)--(A3);
\draw (A6)--(4,1.4);
\draw[very thick] (A1)--(A3);
\draw[very thick] (A4)--(A6);
\draw (-1.039,0.6)--(A1)--(1.039,0.6);
\draw[dotted, thick] (A1)--(0,0.525);
\draw[dotted, thick] (0,1.075)--(0,1.4);
\draw[dotted, thick] (0,-1.965)--(0,-2.2);
\draw[dotted, thick] (4,-1.965)--(4,-2.2);

\foreach \t in {1,2,4,5} \fill[black] (A\t) circle (0.06);

\coordinate [label=below:$V(\cF)$] (a) at (0,-2.3);
\coordinate [label=below:$V(\cG)$] (b) at (4,-2.3);
\end{tikzpicture}
\caption{$V(\cF)$ and $V(\cG)$ in a neighborhood $U_+$ of $P_+$ ($\cF=\cF_1$).}
\label{f_1_m_+}
\end{figure}

\begin{notation}
For a Laurent polynomial $f=\sum_{\bi}c_{\bi}\mathbf{x}^{\bi}\in k[x_1^{\pm1}, \dots, x_n^{\pm1}]$, we define $\coeff_{\bi}(f)=c_{\bi}$ and $v_{\bi}(f)=\val(c_{\bi})$.
\end{notation}

\begin{notation}\label{NOTA}
Let $f, g\in k[x^{\pm1}, y^{\pm1}]$ and $L\in \cSI(f, g)=\cRayfg\cup \cLSfg$.
In the rest of this paper, we use the following notation.
\begin{itemize}
\item
$\Phi_1(L)=\sembunioil$ and $\Phi_2(L)=\sembunjojl$, where $\bil-\bio=\bjl-\bjo$ (see Corollary \ref{il-io}).
\item
An endpoint $P_+\in L$ is a vertex of $\Trop(V(f_1))$ ($f_1\in \{f, g\}$).
Let $(\blo, \bll)$ be $(\bio, \bil)$ (resp. $(\bjo, \bjl)$) if $f_1=f$ (resp. $f_1=g$).
\item
The vertex of the $2$-simplex of $\Delta_{f_1}$ corresponding to $P_+$ are $\blo$, $\bll$ and $\blp$.
Let $\bip:=\bio+(\blp-\blo)$ and $\bjp:=\bjo+(\blp-\blo)$.
\item
If $L\in \cLSfg$, the other endpoint $P_-\in L$ is a vertex of $\Trop(V(f'_1))$ ($f'_1\in \{f, g\}$).
Let $(\blo', \bll')$ be $(\bio, \bil)$ (resp. $(\bjo, \bjl)$) if $f'_1=f$ (resp. $f'_1=g$).
\item
The vertex of the $2$-simplex of $\Delta_{f'_1}$ corresponding to $P_-$ are $\blo'$, $\bll'$ and $\blm'$.
Let $\bim:=\bio+(\blm'-\blo')$ and $\bjm:=\bjo+(\blm'-\blo')$.
\end{itemize}
\end{notation}

By multiplying a unit, we may assume that $f$, $g$ and $L$ further satisfy the following condition $(\P)$:
\begin{itemize}
\item
$\Phi_1(L)=\Phi_2(L)=\sembunioil$.
\vspace{1mm}
\item
$v_{\bio}(f)=v_{\bio}(g)$.
\vspace{1mm}
\item
$v_{\bil}(f)=v_{\bil}(g)$.
\end{itemize}

Furthermore, by applying an affine transformation, multiplying units and changing the variable $x$ to $\coeff_{10}(f)x$, we may assume that $f$, $g$ and $L$ satisfy the following condition $(\P')$:
\begin{itemize}
\item
$\Phi_1(L)=\Phi_2(L)=\sembunioil$.
\vspace{1mm}
\item
$\bio=(0, 0)$, $\bil=(1, 0)$ and $\bip=(0, 1)$.
\vspace{1mm}
\item
$v_{\bio}(f)=v_{\bio}(g)=v_{\bil}(f)=v_{\bil}(g)=0$.
\vspace{1mm}
\item
$P_+=(0, y_+)$ and $P_-=(0, y_-)$.
\end{itemize}

Now we are going to find an element of the ideal $(f, g)$ that is useful in studying $\trop(V(f)\cap V(g))$.
This will be of the form $G=g+h(\mathbf{x}^{\mathbf{v}})f$, where $h\in k[t^{\pm1}]$ is a univariate Laurent polynomial.
The proof of the following lemma gives an algorithm to find this element.

\begin{lemma}\label{lem_f_L}
Let $\de > 0$ be a positive number.
Let $f, g\in k[x^{\pm1}, y^{\pm1}]$ be Laurent polynomials satisfying the following: 
\begin{itemize}
\item $v_{\bio}(f)=v_{\bio}(g)\neq \infty$,\ \ $v_{\bil}(f)=v_{\bil}(g)\neq \infty$.
\item $\bil-\bio$ is primitive.
\item $\mu(f; \sembunioil)>0,\ \mu(g; \sembunioil)>0$ (see Definition \ref{mu}).
\end{itemize}
Then, there exists a Laurent polynomial $h\in k[t^{\pm1}]$ satisfying the following conditions:
\begin{itemize}
\item
For all $i\in \mathbb{Z}$, we have $v_i(h)>i(v_{\bil}(f)-v_{\bio}(f))$.
\item
For the Laurent polynomial $g':=g+h(\mathbf{x}^{\bil-\bio})f$, we have
\begin{eqnarray*}
&v_{\mathbf{i}_0}(g')=v_{\mathbf{i}_0}(g'),& \\
&v_{\mathbf{i}_1}(g')=v_{\mathbf{i}_1}(g'),& \\
&\mu(g'; \sembunioil)> \de.&
\end{eqnarray*}
\end{itemize}
\end{lemma}

\begin{proof}
We can assume that $\mathbf{i}_0=(0, 0)$ and $\mathbf{i}_1=(1, 0)$ by applying an affine transformation.
Then, the statements are only about the coefficients of $x^i$, and we can assume that $f=\sum_{i}c_ix^i, g=\sum_{i}d_ix^i\in k[x^{\pm1}]$.
We can also assume that $c_0=1$ by multiplying a unit.
By changing the variable $x$ to $c_1x$, we may also assume $c_1=1$.
Given a Laurent polynomial $F=\sum_i\alpha_ix^i\in k[x^{\pm1}]$, we define
\begin{eqnarray*}
v(F)&=&\min\{\val(\alpha_i)\},\\
v'(F)&=&\min\{\val(\alpha_i)\ |\ i\neq 0, 1\},\\
v'_+(F)&=&\min\{\val(\alpha_i)\ |\ i>1\},\\
v'_-(F)&=&\min\{\val(\alpha_i)\ |\ i< 0\}.
\end{eqnarray*}
Then, we have $v'(f)>0$ and $v'(g)>0$.
It is sufficient to show that there exists a Laurent polynomial $h\in k[x^{\pm1}]$ with $v(h)>0$ such that for the Laurent polynomial $g':=g+hf$, we have $v_0(g')=v_1(g')=0$ and $v'(g')\geq \de$.
Let $\de_0=v'(f)$ and $\de_1=v'(g)$.
Given a Laurent polynomial $F\in k[x^{\pm1}]$, we define
\begin{eqnarray*}
\Mm(F)&=&\min\{n\in \mathbb{Z}\ |\ n\leq0,\ v_n(F)<\de_0+\de_1\},\\
\Mp(F)&=&\max\{n\in \mathbb{Z}\ |\ n\geq1,\ v_n(F)<\de_0+\de_1\}.
\end{eqnarray*}
Note that $\Mm(F)\leq 0$, $\Mp(F)\geq 1$ and that for Laurent polynomials $F_1, F_2\in k[x^{\pm1}]$, we have
\begin{eqnarray*}
\Mm(F_1+F_2)&\geq& \min\{\Mm(F_1), \Mm(F_2)\},\\
\Mp(F_1+F_2)&\leq& \max\{\Mp(F_1), \Mp(F_2)\}.
\end{eqnarray*}

\begin{claim}\label{calcu}
The following hold.
\begin{itemize}
\item
If $\Mp(g)>1$, then there exist $a\in k$ and $i\in \mathbb{Z}$ such that $\val(a)>0$, $\Mp(g-ax^if)<\Mp(g)$ and $\Mm(g-ax^if)\geq \Mm(g)$.
\item
If $\Mm(g)<0$, then there exist $a\in k$ and $i\in \mathbb{Z}$ such that $\val(a)>0$, $\Mp(g-ax^if)\leq \Mp(g)$ and $\Mm(g-ax^if)> \Mm(g)$.
\end{itemize}
\end{claim}

\begin{proof}
We show the case where $n:=\Mp(g)>1$.
The proof in the case where $\Mm(g)<0$ is similar.
Let $a=d_n$.
Then, we have $\val(a)\geq \de_1 (>0)$.
Let $-ax^{n-1}f=\sum_{i}\alpha_ix^i$ and $g-ax^{n-1}f=\sum_{i}\beta_ix^i$.
Then, we have $\val(\beta_n)=\val(0)=\infty$ and
\begin{eqnarray*}
i< n-1 &\Rightarrow& \val(\alpha_i)\geq\de_0+\de_1,\\
i=n-1, n &\Rightarrow& \val(\alpha_i)=\val(a)\geq \de_1,\\
n<i &\Rightarrow& \val(\alpha_i)\geq\de_0+\de_1.
\end{eqnarray*}
Thus, we have
\begin{eqnarray*}
n\leq i \Rightarrow \val(\beta_i)\geq\de_0+\de_1,
\end{eqnarray*}
and $\Mm(-ax^{n-1}f)=0$.
Hence, we have
\begin{eqnarray*}
\Mp(g-ax^{n-1}f)<n=\Mp(g),
\end{eqnarray*}
and
\begin{eqnarray*}
\Mm(g-ax^{n-1}f)\geq \min\{\Mm(g), \Mm(-ax^{n-1}f)\}\geq \Mm(g).
\end{eqnarray*}
\end{proof}

Note that in the proof of the above claim, we have $\val(\beta_0)=\val(\beta_1)=0$.
From the first bullet in the above claim, we can show by induction on $n=\Mp(g)$ that there exists a Laurent polynomial $h_0\in k[x^{\pm1}]$ with $v(h_0)>0$ such that for the Laurent polynomial $g_1:=g+h_0f$, we have $v_0(g_1)=v_1(g_1)=0$ and $v'_+(g_1)\geq \de_0+\de_1$.
Then, from the second bullet in the above claim, we may ensure that there exists $h_1\in k[x^{\pm1}]$ with $v(h_1)>0$ such that for $g_2:=g_1+h_1f$, we have $v_0(g_2)=v_1(g_2)=0$ and $v'_-(g_2)> \de_0+\de_1$.
It follows that $g_2=g+(h_0+h_1)f$, $v(h_0+h_1)>0$ and $v'(g_2)> \de_0+\de_1$.

Then, by induction on $\max\left\{0, \left[(\de-v'(g))/v'(f)\right]+1\right\}$, we may ensure that there exists $h\in k[x^{\pm1}]$ with $v(h)>0$ such that for $g':=g+hf$, we have $v_0(g')=v_1(g')=0$ and $v'(g')> \de$.
\end{proof}

\begin{definition}\label{def_G_de}
Let $\de > 0$ be a positive number and $f, g\in k[x^{\pm1}, y^{\pm1}]$ Laurent polynomials satisfying the assumption of Lemma \ref{lem_f_L}.
We define $h(\de; g, f; \sembunioil)\in k[t^{\pm1}]$ to be the Laurent polynomial $h\in k[t^{\pm1}]$ obtained by the algorithm in the proof of Lemma \ref{lem_f_L}.
We also define $G(\de; g, f; \sembunioil)\in k[x^{\pm1}, y^{\pm1}]$ by
\begin{eqnarray*}
G(\de; g, f; \sembunioil):=g_{\de}-\frac{\coeff_{\bil}(g_{\de})}{\coeff_{\bil}(f_{\de})}f_{\de},
\end{eqnarray*}
where
\begin{eqnarray*}
g_{\de}:=g+h(\de; g, f; \sembunioil)(\mathbf{x}^{\bil-\bio})f,\\
f_{\de}:=f+h(\de; f, f; \sembunioil)(\mathbf{x}^{\bil-\bio})f.
\end{eqnarray*}
\end{definition}

More generally, we define the following set.

\begin{definition}\label{H_Elim}
Let $f, g\in k[x^{\pm1}, y^{\pm1}]$ be Laurent polynomials, $L\in \cSIfg$ a ray or a line segment and $\de > 0$ a positive number.
Then, we define $\Hf \subset k[x^{\pm1}, y^{\pm1}]^4$ and $\Elim \subset k[x^{\pm1}, y^{\pm1}]$ by
\begin{eqnarray*}
\Hf=\left\{ (h_{1}, h_{2}, h_{3}, h_{4})\ \middle| 
\begin{array}{l}
\mu\left(f+h_{1}f+h_{2}g; \Phi_1(L)\right)>\de,\\
\mu\left(g+h_{3}f+h_{4}g; \Phi_2(L)\right)>\de,\\
\text{and, for any}\ P\in L,\\
\trop(h_{1})(P)<0,\\
\trop\left(h_2\mathbf{x}^{\bjo-\bio}\right)(P)<v_{\bio}(f)-v_{\bjo}(g),\\
\trop\left(h_3\mathbf{x}^{\bio-\bjo}\right)(P)<v_{\bjo}(g)-v_{\bio}(f),\\
\trop(h_{4})(P)<0
\end{array}
\right\},
\end{eqnarray*}
where $\Phi_1(L)=\sembunioil$ and $\Phi_2(L)=\sembunjojl$ are endowed with the same orientation, and
\begin{eqnarray*}
\Elim=\left\{ G\ \middle| 
\begin{array}{l}
\exists (h_{1}, h_{2}, h_{3}, h_{4})\in \Hf\ \text{s.t.}\\
G=g'-\frac{\coeff_{\bjl}(g')}{\coeff_{\bil}(f')}\mathbf{x}^{\bjl-\bil}f',\\
\text{where}\ f'=f+h_{1}f+h_{2}g\ \text{and}\ g'=g+h_{3}f+h_{4}g
\end{array}
\right\}.
\end{eqnarray*}
\end{definition}

\begin{remark}
Let $f, g\in k[x^{\pm1}, y^{\pm1}]$ be Laurent polynomials and $L\in \cSIfg$ a ray or a line segment satisfying the condition $(\P)$.
Let $\de > 0$ be a positive number.
Then, we have $(h(\de; f, f; \sembunioil)(x), 0, h(\de; g, f; \sembunioil)(x), 0)\in \Hf$ and $G(\de; g, f; \sembunioil)\in \Elim$.
\end{remark}

To compare the tropicalizations of $V(f)$, $V(g)$ and $V(G)$ for $G\in \Elim$, we use the following lemma.

\begin{lemma}\label{lem_f_g}
Let $f, g\in k[x^{\pm1}, y^{\pm1}]$ be Laurent polynomials and $L\in \cSIfg$ a ray or a line segment.
Let $\de > 0$ be a positive number and $(h_{1}, h_{2}, h_{3}, h_{4})\in \Hf$.
Then, the following hold.
\begin{enumerate}
\item
For any $\mathbf{i}\in \mathbb{Z}^2$ and $P\in L$, we have
\begin{eqnarray*}
\tau\left(h_{1}f+h_{2}g; \mathbf{i}; P\right)&<& \tau(f; \bio; P),\\
\tau\left(h_{3}f+h_{4}g; \mathbf{i}; P\right)&<& \tau(g; \bjo; P).
\end{eqnarray*}
\item
We have
\begin{eqnarray*}
&v_{\bio}\left(f+h_{1}f+h_{2}g\right)=v_{\bio}(f),&\\
&v_{\bil}\left(f+h_{1}f+h_{2}g\right)=v_{\bil}(f),&\\
&v_{\bjo}\left(g+h_{3}f+h_{4}g\right)=v_{\bjo}(g),&\\
&v_{\bjl}\left(g+h_{3}f+h_{4}g\right)=v_{\bjl}(g).&
\end{eqnarray*}
\end{enumerate}
\end{lemma}

\begin{proof}
By replacing $g$ by $(\coeff_{\bio}(f)/\coeff_{\bjo}(g))\mathbf{x}^{\bio-\bjo}g$, we may assume that $f$, $g$ and $L$ satisfy $(\P)$.
Since inequalities about $\tau$ does not change by coordinate change, we may further assume that $f$, $g$ and $L$ satisfy the condition $(\P')$.
Then, the statements (1) and (2) clearly hold.

\end{proof}

\begin{lemma}\label{lem_f_g_G}
Let $f, g\in k[x^{\pm1}, y^{\pm1}]$ be Laurent polynomials and $L\in \cSIfg$ a ray or a line segment.
For any $\de > 0$ and $G\in \Elim$, the following hold.
\[
V(f, G)\cap \tropin(L)=V(g, G)\cap \tropin(L)=V(f, g)\cap \tropin(L).
\]
\end{lemma}

\begin{proof}
We may assume that $f$, $g$ and $L$ satisfy the condition $(\P')$.
Let $\de>0$ and $G\in \Elim$.
We show $V(f, G)\cap \tropin(L)=V(f, g)\cap \tropin(L)$.
We can show $V(g, G)\cap \tropin(L)=V(f, g)\cap \tropin(L)$ in the same way.
There exists $(h_{1}, h_{2}, h_{3}, h_{4})\in \Hf$ such that $G=g'-(d'/c')f'$, where $f'=f+h_{1}f+h_{2}g$, $g'=g+h_{3}f+h_{4}g$, $c'=\coeff_{\bil}(f')$ and $d'=\coeff_{\bil}(g')$.
Thus, we have
\begin{eqnarray*}
V(f, G)=V\left(f, \left(1+h_{4}-\frac{d'}{c'}h_{2}\right)g\right).
\end{eqnarray*}
Here, by Lemma \ref{lem_f_g} (2), we have $\val(d'/c')=0$.
Combined with $\trop(h_{2})(P)<0=\trop(1)(P)$ and $\trop(h_{4})(P)<0=\trop(1)(P)$ for all $P\in L$, it follows that
\[
V\left(1+h_{4}-\frac{d'}{c'}h_{2}\right)\cap \tropin(L)=\emptyset.
\]
Therefore, we have
\begin{eqnarray*}
V(f, G)\cap \tropin(L)&=&V(f)\cap V\left(\left(1+h_{4}-\frac{d'}{c'}h_{2}\right)g\right)\cap \tropin(L)\\
&=&V(f)\cap V(g)\cap \tropin(L)\\
&=&V(f, g)\cap \tropin(L).
\end{eqnarray*}
\end{proof}

\begin{notation}
We write $\mathbf{e}_1=(1, 0), \mathbf{e}_2=(0, 1)\in \mathbb{R}^2$ for the standard basis.
\end{notation}

\begin{lemma}\label{tatehoukou}
Let $h\in k[x^{\pm1}, y^{\pm1}]$, $\bj \in \mathbb{Z}^2$, $P'\in \mathbb{R}^2$, $\mathbf{v}_1 \in \mathbb{R}^2\setminus \{\mathbf{0}\}$ and $\mathbf{v}_2 \in \mathbb{R}^2\setminus \Aff(\mathbf{v}_1)$.
Let $\mathbf{w}\in \mathbb{R}^2\setminus \{\mathbf{0}\}$ a normal vector of $\mathbf{v}_1$ such that $\mathbf{w} \cdot \mathbf{v}_2<0$.
Assume that for a lattice point $\mathbf{i}\neq \bj$ in the half plane $\bj+\mathbb{R}\mathbf{v}_1+\mathbb{R}_{\geq 0}\mathbf{v}_2$, we have $\tau(h; \bj; P')>\tau(h; \bi; P')$.
Then, for all $P\in P'+\mathbb{R}_{\geq 0}\mathbf{w}$, we have
\[
\tau(h; \bj; P)> \tau(h; \mathbf{i}; P).
\]
\end{lemma}

\begin{proof}
Let $P\in P'+\mathbb{R}_{\geq 0}\mathbf{w}$.
Then, there exists a non-negative number $r\geq0$ such that $P=P'+r\mathbf{w}$, and hence, we have
\begin{eqnarray*}
\tau(h; \bj; P)-\tau(h; \bi; P)=(\tau(h; \bj; P')-\tau(h; \bi; P'))+r(\bj-\bi) \cdot \mathbf{w}>0.
\end{eqnarray*}
\end{proof}

\begin{remark}\label{claim_f_1}
Let $f, g\in k[x^{\pm1}, y^{\pm1}]$ and $L\in \cSI(f, g)$ satisfy $(\P)$.
Then, for an endpoint $P_+\in L$, by Lemma \ref{Lnohashi}, we have
\begin{eqnarray*}
\tau(f_1; \bip; P_+)>\tau(f_2; \bip; P_+),
\end{eqnarray*}
where $\{f_1, f_2\}=\{f, g\}$ and $\trop(V(f_1))$ has a vertex at $P_+$.
Therefore, either $v_{\bip}(f)<v_{\bip}(g)$ and $f_1=f$ or $v_{\bip}(f)>v_{\bip}(g)$ and $f_1=g$, and hence,
\[
v_{\bip}(f_1)=\min\{v_{\bip}(f), v_{\bip}(g)\}.
\]
\end{remark}

\begin{lemma}\label{claim_e_m_+}
Let $f, g\in k[x^{\pm1}, y^{\pm1}]$ and $L\in \cSI(f, g)$.
Let $\de > 0$ be a positive number and $G\in \Elim$.
Then, the following hold.
\begin{enumerate}
\item If $f$, $g$ and $L$ satisfy $(\P)$, then $v_{\bip}(G)=v_{\bip}(f_1)\ (=\min \{v_{\bip}(f), v_{\bip}(g)\})$.
\item Assume that a lattice point $\mathbf{i}\neq \bjp$ is in the half plane $\bjp+\mathbb{R}(\bjl-\bjo)+\mathbb{R}_{\geq 0}(\bjp-\bjo)$.
Then, for all $P\in L$, we have
\[
\tau(G; \mathbf{i}; P)<\tau(G; \bjp; P).
\]
\item Assume that $L\in \cLSfg$ and $\mathbf{i}\neq \bjm$ is a lattice point in the half plane $\bjm +\mathbb{R}(\bjl-\bjo)+\mathbb{R}_{\geq 0}(\bjm-\bjo)$.
Then, for any point $P\in L$, we have
\[
\tau(G; \mathbf{i}; P)<\tau(G; \bjm; P).
\]
\end{enumerate}
\end{lemma}

\begin{proof}
Let $G=g'-\frac{\coeff_{\bil}(g')}{\coeff_{\bil}(f')}f'$, where $f'=f+h_{1}f+h_{2}g$ and $g'=g+h_{3}f+h_{4}g$ with $(h_1, h_2, h_3, h_4)\in \Hf$.
To show (1), first note that, by Lemma \ref{lem_f_g}, we have
\begin{eqnarray*}
\tau(f; \bio; P_+)> \tau(h_{1}f+h_{2}g; \bip; P_+),\ \tau(h_{3}f+h_{4}g; \bip; P_+).
\end{eqnarray*}
Let $U_+$ be a sufficiently small neighborhood of $P_+$ and $P\in U_+\cap (\Aff(L)\setminus L)$ a point.
Then, we have
\begin{eqnarray*}
\tau(f; \bio; P)> \tau(h_{1}f+h_{2}g; \bip; P),\ \tau(h_{3}f+h_{4}g; \bip; P).
\end{eqnarray*}
Combined with Lemma \ref{Lnohashi}, this implies
\begin{eqnarray}
\tau(f_1; \bip; P)> \tau(h_{1}f+h_{2}g; \bip; P),\ \tau(h_{3}f+h_{4}g; \bip; P).
\end{eqnarray}
Now, by Lemma \ref{lem_f_g} (2), we have $\val(\coeff_{\bil}(g')/\coeff_{\bil}(f'))=0$ and
\begin{eqnarray*}
G=g-\frac{\coeff_{\bil}(g')}{\coeff_{\bil}(f')}f+h_{3}f+h_{4}g-\frac{\coeff_{\bil}(g')}{\coeff_{\bil}(f')}\left(h_{1}f+h_{2}g\right).
\end{eqnarray*}
Therefore, we have $\tau(G; \bip; P)=\tau(f_1; \bip; P)$ by (6) and Lemmas \ref{tau(f+g)} and \ref{Lnohashi}.
Thus, we have $v_{\bip}(G)=v_{\bip}(f_1)$.

Next, let us show (2).
(3) follows from (2) by symmetry.
We may assume that $f$, $g$ and $L$ satisfy the condition $(\P')$.
Let $P'\in L$.
By Lemmas \ref{lem_f_g} and \ref{tatehoukou}, we have
\begin{eqnarray*}
\max\{\tau\left(h_{1}f+h_{2}g; \mathbf{i}; P'\right), \tau\left(h_{3}f+h_{4}g; \mathbf{i}; P'\right)\}< \tau(f_1; \bip; P').
\end{eqnarray*}
Since $\mathbf{i}\neq \bip$, by Lemmas \ref{Lnohashi} and \ref{tatehoukou}, we have
\begin{eqnarray*}
\max\{\tau(f; \mathbf{i}; P'), \tau\left(g; \mathbf{i}; P'\right)\}<\tau(f_1; \bip; P').
\end{eqnarray*}
Therefore, by Lemma \ref{tau(f+g)}, we have
\begin{eqnarray*}
\tau(G; \mathbf{i}; P')<\tau(f_1; \bip; P').
\end{eqnarray*}
Since $v_{\bip}(f_1)=v_{\bip}(G)$, we have
\[
\tau(f_1; \bip; P')=\tau(G; \bip; P'),
\]
and hence, we have
\[
\tau(G; \mathbf{i}; P')<\tau(G; \bip; P').
\]
\end{proof}

\begin{corollary}\label{Ljounobip}
Let $f, g\in k[x^{\pm1}, y^{\pm1}]$ and $L\in \cSI(f, g)$.
Let $\de > 0$ be a positive number and $G\in \Elim$.
Then, for a point $P_1=P_++r_1\mathbf{v}_{P_+, L}$ (see Notation \ref{v_P,L}), where $r_1\in \mathbb{R}$, we have
\begin{eqnarray*}
\tau(G; \bjp; P)=\tau(g; \bjo; P_+)-r_1.
\end{eqnarray*}
If $L\in \cLSfg$, then for a point $P_2=P_-+r_2\mathbf{v}_{P_-, L}$ $(r_2\in \mathbb{R})$, we have
\begin{eqnarray*}
\tau(G; \bjm; P)=\tau(g; \bjo; P_-)-r_2.
\end{eqnarray*}
\end{corollary}

\begin{proof}
We may assume that $f$, $g$ and $L$ satisfy the condition $(\P')$.
By Lemmas \ref{Lnohashi} and \ref{claim_e_m_+} (1), for the point $P_1=(0, y_+-r_1)\in \mathbb{R}^2$, we have
\begin{eqnarray*}
\tau(G; \bip; P_1)&=&\tau(f_1; \bip; P_1)\\
&=&\tau(f_1; \bip; P_+)+r_1\bip \cdot (-\mathbf{e}_2)\\
&=&\tau(g; (0, 0); P_+)-r_1.
\end{eqnarray*}
Similarly, if $L\in \cLSfg$, we have $\tau(G; \bim; P_2)=\tau(g; (0, 0); P_-)-r_2$.
\end{proof}

\begin{notation}
For $L\in \cSIfg$ and $\de> 0$, we define
\[
L^{\de}_+=L\cap \{P_++r\mathbf{v}_{P_+, L}\ |\ 0\leq r<\de \}.
\]
If $L\in \cLSfg$, we also define
\[
L^{\de}_-=L\cap \{P_-+r\mathbf{v}_{P_-, L}\ |\ 0\leq r<\de \}.
\]
\end{notation}

\begin{lemma}\label{claim_three_index_ray2}
Let $f, g\in k[x^{\pm1}, y^{\pm1}]$ and $L\in \cSI(f, g)$.
Let $\de > 0$ be a positive number and $G\in \Elim$.
Then, the following hold.
\medbreak
\begin{enumerate}
\item $\forall n\in \mathbb{Z}\setminus \{0\},\ \forall P\in L,\ \tau(G; \bjo+n(\bjl-\bjo); P)<\tau(g; \bjo; P)-\de$.
\medbreak
\item If $L\in \cRayfg$, then for a sufficiently small neighborhood $U_+$ of $L^{\de}_+$, we have $\{\mathbf{i}\in \mathbb{Z}^2\ |\ \exists P\in U_+\text{ s.t. }\tau(G; \mathbf{i}; P)=\trop(G)(P)\} \subset \{\bjo, \bjp\}$.
\medbreak
\item If $L\in \cLSfg$, then for sufficiently small neighborhoods $U_+$ of $L^{\de}_+$ and $U_-$ of $L^{\de}_-$, we have $\{\mathbf{i}\in \mathbb{Z}^2\ |\ \exists P\in U_+\cup U_-\text{ s.t. }\tau(G; \mathbf{i}; P)=\trop(G)(P)\} \subset \{\bjo, \bjp, \bjm\}$.
\end{enumerate}
\end{lemma}

\begin{proof}
We may assume that $f$, $g$ and $L$ satisfy the condition $(\P')$.
Let us show (1).
Let $\de> 0$ and $P\in L$.
Since we have
\begin{eqnarray*}
\forall n\in \mathbb{Z}\setminus \{0, 1\},\ \mu(f'; \sembunioil)>\de\ \text{ and }\ \mu(g'; \sembunioil)>\de,
\end{eqnarray*}
and $G=g'-(\coeff_{\bil}(g')/\coeff_{\bil}(f'))f'$, we have
\begin{eqnarray*}
\forall n\in \mathbb{Z}\setminus \{0, 1\},\ v_{n0}(G)>\de,
\end{eqnarray*}
i.e.,
\[
\forall n\in \mathbb{Z}\setminus \{0, 1\},\ \tau(G; (n, 0); P)=-v_{n0}(G)<-\de.
\]
Noting that
\begin{eqnarray*}
\tau(G; (1, 0); P)=-v_{10}(G)=-\infty<-\de,
\end{eqnarray*}
we see that
\[
\forall n\in \mathbb{Z}\setminus \{0\},\ \tau(G; (n, 0); P)<-\de.
\]

Let us show (2).
(3) follows from (2) symmetry.
Since the number of the terms of $G$ is finite and each term of $\trop(G)$ is a continuous and piecewise linear map, it is sufficient to show that $\{\mathbf{i}\in \mathbb{Z}^2\ |\ \exists P\in L^{\de}_+\text{ s.t. }\tau(G; \mathbf{i}; P)=\trop(G)(P)\} \subset \{\bio, \bip\}$.
By the assumption that the three vertices of the corresponding $2$-simplex of $\Delta_{f_1}$ are $\bio$, $\bil$ and $\bip=(0, 1)$, we have $L=P_++\mathbb{R}_{\geq 0}(-\mathbf{e}_2)$, and the condition $\Phi_1(L)=\Phi_2(L)=\sembunioil$ implies
\[
\forall (i, j)\in \mathbb{Z}^2, j<0\Rightarrow c_{ij}=d_{ij}=0.
\]
Combined with Lemma \ref{claim_e_m_+} (2), it follows that
\[
\{\mathbf{i}\in \mathbb{Z}^2\ |\ \exists P\in L\text{ s.t. }\tau(G; \mathbf{i}; P)=\trop(G)(P)\} \subset \mathbb{Z}\mathbf{e}_1\cup \{\bip\}.
\]
Then, by (1) and Corollary \ref{Ljounobip}, for any $n\in \mathbb{Z}\setminus \{0\}$ and any $P\in L^{\de}_+$, if we write $P=P_++r(-\mathbf{e}_2)$ ($0\leq r<\de$), we have
\[
\tau(G; (n, 0); P)<\tau(g; (0, 0); P)-\de=\tau(g; (0, 0); P_+)-\de<\tau(G; \bip; P),
\]
and hence, the assertion holds.
\end{proof}

\section{Proofs of the main theorems}

The following proposition gives us a way of determining $\trop(V(f, g))\cap L$ for $L\in \cSIfg$, and will be the main tool in finding a polynomial that realizes the desired intersection.
Note that the points in $\trop(V(f, g))$ are equipped with the multiplicities coming from the intersection multiplicities of $V(f)\cap V(g)$.

\begin{proposition}\label{prop_d_00}
Let $f, g\in k[x^{\pm1}, y^{\pm1}]$ be Laurent polynomials.
\begin{enumerate}
\item Let $L\in \cRayfg$ be a ray.
Then, for $\de > 0$ and some (or any) $G\in \Elim$, we have
\[
\trop(V(f, g))\cap L^{\de}_+=
\begin{cases}
\{P_++(v_{\bjo}(G)-v_{\bjo}(g))\mathbf{v}_{P_+, L}\} & (v_{\bjo}(G)-v_{\bjo}(g)< \de),\\
\emptyset & (v_{\bjo}(G)-v_{\bjo}(g)\geq \de).
\end{cases}
\]
In particular, $\trop(V(f, g))\cap L=\emptyset$ if and only if for any $\de >0$ and $G\in \Elim$, we have $v_{\bjo}(G)-v_{\bjo}(g)\geq \de$.
\item Let $L\in \cLSfg$ be a line segment.
Let $l=\dist(P_+, P_-)$.
Then, for $\de > 0$ and $G\in \Elim$, we have
\begin{eqnarray*}
&&\trop(V(f, g))\cap (L^{\de}_+\cup L^{\de}_-)\\
&&=
\begin{cases}
\left\{
\begin{array}{l}
P_++(v_{\bjo}(G)-v_{\bjo}(g))\mathbf{v}_{P_+, L},\\
P_-+(v_{\bjo}(G)-v_{\bjo}(g))\mathbf{v}_{P_-, L}
\end{array}
\right\}
& \left(v_{\bjo}(G)-v_{\bjo}(g)<\min\left\{ \frac{l}{2}, \de\right\} \right),\\
\left\{ \frac{P_++P_-}{2}\right\}\ (\text{multiplicity}=2) & \left(\frac{l}{2}\leq v_{\bjo}(G)-v_{\bjo}(g)\ \text{and}\ \frac{l}{2}< \de\right),\\
\emptyset & \left(\de\leq \min \left\{ \frac{l}{2}, v_{\bjo}(G)-v_{\bjo}(g)\right\} \right).
\end{cases}
\end{eqnarray*}
In particular, if $\de>l/2$, then we have
\begin{eqnarray*}
\trop(V(f, g))\cap L=
\begin{cases}
\left\{
\begin{array}{l}
P_++(v_{\bjo}(G)-v_{\bjo}(g))\mathbf{v}_{P_+, L},\\
P_-+(v_{\bjo}(G)-v_{\bjo}(g))\mathbf{v}_{P_-, L}
\end{array}
\right\}
& \left(v_{\bjo}(G)-v_{\bjo}(g)<\frac{l}{2} \right),\\
\left\{ \frac{P_++P_-}{2}\right\}\ (\text{multiplicity}=2) & \left(\frac{l}{2}\leq v_{\bjo}(G)-v_{\bjo}(g)\right).
\end{cases}
\end{eqnarray*}
\end{enumerate}
\end{proposition}

\begin{proof}
We may assume that $f$, $g$ and $L$ satisfy the condition $(\P')$.
Let us show (1).
Let $\de>0$ and $G\in \Elim$.
Let $U_+$ be a sufficiently small neighborhood of $L^{\de}_+$.
By Corollary \ref{Ljounobip}, for a point $(0, y)\in \mathbb{R}^2$, we have $\tau(G; \bip; (0, y))=y-y_+$.

Assume that $v_{00}(G)< \de$.
Then, noting that $\tau(G; \bio; (0, y))=-v_{00}(G)$, we have
\begin{eqnarray*}
y_+-v_{00}(G)<y&\Rightarrow& \tau(G; \bip; (0, y))> \tau(G; \bio; (0, y)),\\
y=y_+-v_{00}(G)&\Rightarrow& \tau(G; \bip; (0, y))=\tau(G; \bio; (0, y)),\\
y<y_+-v_{00}(G)&\Rightarrow& \tau(G; \bip; (0, y))<\tau(G; \bio; (0, y)).
\end{eqnarray*}
Combined with Lemma \ref{claim_three_index_ray2} (2), it follows that $\Trop(V(G))\cap U_+\cap (x=0)=\{(0, y_+-v_{00}(G))\}$.
Note that we consider $U_+$ to deal with the case where $v_{00}(G)=0$.
Then, for $f_2=f$ or $g$, we have
\begin{eqnarray*}
\Trop(V(f_2))\cap \Trop(V(G))\cap U_+=\{(0, y_+-v_{00}(G))\}\ \ (\text{see Figure \ref{Fig_G_delta}}).
\end{eqnarray*}
Hence, $\{(0, y_+-v_{00}(G))\}$ is an isolated point of $\Trop(V(f_2)) \cap \Trop(V(G))$.
Note that the intersection multiplicity of $\Trop(V(f_2))$ and $\Trop(V(G))$ at $(0, y_+-v_{00}(G))$ is $1$ (see Figure \ref{Fig_G_delta}).
Hence, by Theorem \ref{OR6.13}, there exists a unique point $\mathbf{x}\in V(f_2, G)$ such that $\trop(\mathbf{x})=(0, y_+-v_{00}(G))$.
Thus, by Lemma \ref{lem_f_g_G}, we have
\[
\trop(V(f, g))\cap L^{\de}_+=\trop(V(f_2, G))\cap L^{\de}_+=\{(0, y_+-v_{00}(G))\}.
\]

If $v_{00}(G)\geq \de$ and $y\in (y_+-\de, y_+]$, then $\{\tau(G; \bi; (0, y))\ |\ \bi \in \mathbb{Z}^2\}$ takes the maximal value only at $\bi=\bip$, and we have $\Trop(V(G))\cap L^{\de}_+=\emptyset$.
In this case, by Lemma \ref{lem_f_g_G}, it follows that $\trop(V(f, g))\cap L^{\de}_+\subset \Trop(V(G))\cap L^{\de}_+=\emptyset$.

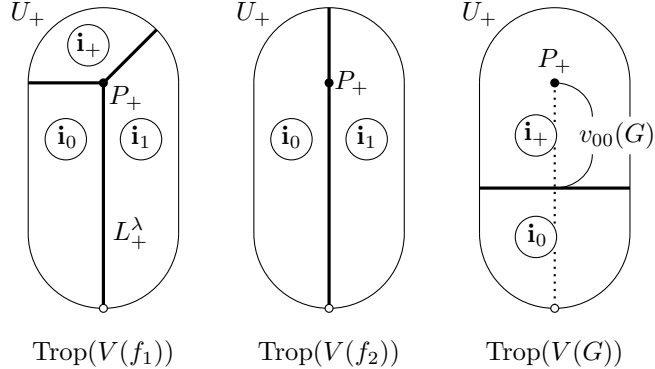
\begin{figure}[H]
\centering
\begin{tikzpicture}
\draw (-0.5,-0.75) circle [radius=0.275];
\draw (0.5,-0.75) circle [radius=0.275];
\draw (-0.2,0.5) circle [radius=0.275];

\draw (2.5,-0.75) circle [radius=0.275];
\draw (3.5,-0.75) circle [radius=0.275];

\draw (5.75,-2.05) circle [radius=0.275];
\draw (5.75,-0.7) circle [radius=0.275];
\draw (6,0) to [out=0, in=90] (6.5,-0.45);
\draw (6,-1.4) to [out=0, in=-90] (6.5,-0.95);
\coordinate [label=right:$v_{00}(G)$] (x1) at (6.2,-0.7);

\draw (4,0)--(4,-2);
\draw (5,0)--(5,-2);
\coordinate (A11) at (3,0);
\coordinate [label=right:$P_+$] (w1) at (2.95,0);
\coordinate (A12) at (3,0.3);
\coordinate (A13) at (3,-3);
\draw[very thick] (A13)--(3,1);
\coordinate [label=above:$\bio$] (a2) at (2.5,-0.99);
\coordinate [label=above:$\bil$] (a3) at (3.5,-0.99);

\draw (1,0) arc (0:180:1);
\draw (4,0) arc (0:180:1);
\draw (7,0) arc (0:180:1);
\draw (1,0)--(1,-2);
\draw (-1,0)--(-1,-2);
\draw (7,0)--(7,-0.45);
\draw (7,-0.95)--(7,-2);
\draw (2,0)--(2,-2);

\draw (1,-2) arc (0:-180:1);
\draw (4,-2) arc (0:-180:1);
\draw (7,-2) arc (0:-180:1);

\draw[very thick] (5,-1.4)--(7,-1.4);

\coordinate (A1) at (0,0);
\coordinate [label=right:$P_+$] (B1) at (-0.05,-0.2);
\coordinate (A2) at (0,0.3);
\coordinate (A3) at (0,-3);
\coordinate (A4) at (6,0);
\coordinate [label=above:$P_+$] (C1) at (6,0);
\coordinate (A5) at (6,0.3);
\coordinate (A6) at (6,-3);
\coordinate [label=right:$L^{\de}_+$] (A7) at (0,-2);

\coordinate [label=below:$U_+$] (B4) at (-1,1.2);
\coordinate [label=below:$U_+$] (B5) at (2,1.2);
\coordinate [label=below:$U_+$] (B5) at (5,1.2);

\coordinate [label=above:$\bip$] (a1) at (-0.19,0.22);
\coordinate [label=above:$\bio$] (a2) at (-0.48,-0.99);
\coordinate [label=above:$\bil$] (a3) at (0.5,-0.99);

\coordinate [label=above:$\bio$] (a4) at (5.77,-2.29);
\coordinate [label=above:$\bip$] (a5) at (5.76,-0.97);

\draw[very thick] (A1)--(A3);
\draw (A1)--(A3);
\draw[very thick] (-1,0)--(A1)--(0.709,0.709);
\draw[dotted, thick] (A4)--(6,-0.58);
\draw[dotted, thick] (6,-0.86)--(6,-1.875);
\draw[dotted, thick] (6,-2.25)--(6,-3);

\foreach \t in {1,4,11} \fill[black] (A\t) circle (0.06);
\coordinate [label=below:$\Trop(V(f_1))$] (a) at (0,-3.3);
\coordinate [label=below:$\Trop(V(f_2))$] (a) at (3,-3.3);
\coordinate [label=below:$\Trop(V(G))$] (b) at (6,-3.3);
\fill[black] (0,-3) circle (0.06);
\fill[black] (3,-3) circle (0.06);
\fill[black] (6,-3) circle (0.06);
\fill[white] (0,-3) circle (0.045);
\fill[white] (3,-3) circle (0.045);
\fill[white] (6,-3) circle (0.045);
\end{tikzpicture}
\caption{$\Trop(V(f_1))$, $\Trop(V(f_2))$ and $\Trop(V(G))$ in a neighborhood $U_+$ of $L^{\de}_+$.}
\label{Fig_G_delta}
\end{figure}

Let us show (2).
Let $\de>0$ and $G\in \Elim$.
Note that $L=\{(0, y)\ |\ y_-\leq y\leq y_+\}$, $l=y_+-y_-$ and that by Corollary \ref{Ljounobip}, we have
\begin{eqnarray*}
\tau(G; \bip; (0, y))=y-y_+,\ \ \tau(G; \bim; (0, y))=y_--y.
\end{eqnarray*}
First, consider the case where $v_{00}(G)<\min\{l/2, \de\}$.
Let $y_1:=y_+-v_{00}(G)$.
Then, we have
\begin{eqnarray*}
y_1<y&\Rightarrow& \tau(G; \bip; (0, y))> \tau(G; \bio; (0, y))> \tau(G; \bim; (0, y)),\\
y=y_1&\Rightarrow& \tau(G; \bip; (0, y))=\tau(G; \bio; (0, y))>\tau(G; \bim; (0, y)),\\
y_+-\frac{l}{2}< y<y_1&\Rightarrow& \tau(G; \bio; (0, y))> \tau(G; \bip; (0, y))> \tau(G; \bim; (0, y)).
\end{eqnarray*}
Combined with Lemma \ref{claim_three_index_ray2} (3), it follows that $V(\trop(G))\cap L^{\de}_+\cap L^{\frac{l}{2}}_+=\{(0, y_1)\}$, and in the same way as in (1), we have
\[
\trop(V(f, g))\cap L^{\de}_+\cap L^{\frac{l}{2}}_+=\{(0, y_+-v_{00}(G))\}.
\]
Similarly, we have 
\[
\trop(V(f, g))\cap L^{\de}_-\cap L^{\frac{l}{2}}_-=\{(0, y_-+v_{00}(G))\}.
\]
By considering the intersection multiplicity, we have
\[
\trop(V(f, g))\cap (L^{\de}_+\cup L^{\de}_-)=\{(0, y_+-v_{00}(G)), (0, y_-+v_{00}(G))\}
\]
(in fact, this is equal to $\trop(V(f, g))\cap L$).

Next, consider the case where $l/2\leq v_{00}(G)$ and $l/2<\de$.
Then, we have
\begin{eqnarray*}
\frac{y_++y_-}{2}<y&\Rightarrow& \tau(G; \bip; (0, y))> \tau(G; \bim; (0, y)),\ \tau(G; \bio; (0, y)),\\
y=\frac{y_++y_-}{2}&\Rightarrow& \tau(G; \bip; (0, y))=\tau(G; \bim; (0, y))\geq \tau(G; \bio; (0, y)),\\
y<\frac{y_++y_-}{2}&\Rightarrow& \tau(G; \bim; (0, y))> \tau(G; \bip; (0, y)),\ \tau(G; \bio; (0, y)).
\end{eqnarray*}
Combined with Lemma \ref{claim_three_index_ray2} (3), it follows that
\[
V(\trop(G))\cap (L^{\de}_+\cup L^{\de}_-)=V(\trop(G))\cap L=\left\{\left(0, \frac{y_++y_-}{2}\right)\right\}.
\]
and in the same way as in (1), we have
\[
\trop(V(f, g))\cap (L^{\de}_+\cup L^{\de}_-)=\left\{\left(0, \frac{y_++y_-}{2}\right)\right\}.
\]
By Theorem \ref{OR6.13}, the multiplicity is $2$.

Finally, consider the case where $\de \leq \min\{l/2, v_{00}(G)\}$.
Here, we have
\begin{eqnarray*}
y_+-\de<y&\Rightarrow& \tau(G; \bip; (0, y))> \tau(G; \bim; (0, y)),\ \tau(G; \bio; (0, y)),\\
y<y_-+\de&\Rightarrow& \tau(G; \bim; (0, y))> \tau(G; \bip; (0, y)),\ \tau(G; \bio; (0, y)).
\end{eqnarray*}
Combined with Lemma \ref{claim_three_index_ray2} (3), it follows that
\[
V(\trop(G))\cap (L^{\de}_+\cup L^{\de}_-)=\emptyset,
\]
and hence, $\trop(V(f, g))\cap (L^{\de}_+\cup L^{\de}_-)=\emptyset$.

Thus, we conclude the proof of Proposition \ref{prop_d_00}.
\end{proof}

The following corollary is immediate.

\begin{corollary}\label{raynihaikko}
Let $f, g\in k[x^{\pm1}, y^{\pm1}]$ be Laurent polynomials and $L\in \cRay(f, g)$ a ray.
Then, there is at most one point, counted with multiplicity, in the intersection $\trop(V(f, g))\cap L$.
\end{corollary}

The following corollary shows a special case of the main theorems where $\cSI'$ consists of one element.

\begin{corollary}\label{cor_d_00}
Let $f$ and $g=\sum_{i, j}d_{ij}x^iy^j$ be Laurent polynomials in $k[x^{\pm1}, y^{\pm1}]$ and $D$ a divisor satisfying the condition $(*)$ in Definition \ref{star}.
Let $L\in \cSI(f, g)$ be a ray or a line segment and let $\Phi_2(L)=\sembunjojl$.
Then there exists an element $\tilde{d}_{\bjo}\in k$ such that if we set $g':=g-d_{\bjo}{\mathbf{x}}^{\bjo}+\tilde{d}_{\bjo}{\mathbf{x}}^{\bjo}$, we have $\trop(g)=\trop(g')$ and $\trop(V(f, g'))|_L=D|_L$.
\end{corollary}

\begin{proof}
We will show the statement in the case where $L\in \cRay(f, g)$, and the proof in the case where $L\in \cLSfg$ is similar.
We may assume that $f$, $g$, $L$ and an endpoint $P_+:=(0, y_+)\in L$ satisfy the condition $(\P')$.
Let $P_1=(0, y_+-\kappa)$ ($\kappa \geq 0$) be the intersection point of $D$ and $L$.
Recall that we are using Notation \ref{NOTA} and $P_+$ is a vertex of $\Trop(V(f_1))$.
Since we have $\tau(f_1; \bip; P_+)=\tau(f_1; \bio; P_+)=0$ and $P_1=P_+-\kappa \mathbf{e}_2$,
we have
\begin{eqnarray*}
\tau(f_1; \bip; P_1)=\tau(f_1; \bip; P_+)-\kappa (\bip \cdot \mathbf{e}_2)=-\kappa.
\end{eqnarray*}
Thus, we have
\[
\kappa=-\tau(f_1; \bip; P_1)=v_{\bip}(f_1)-\bip \cdot P_1.
\]
Since the coordinates of $P_1$ are assumed to belong to the value group, there exists an $\alpha\in k^*$ such that $\val(\alpha)=v_{\bip}(f_1)-\bip \cdot P_1=\kappa$.
Let $\de>\kappa$, $h_{\de}:=h(\de; f, f; \sembunioil), h'_{\de}:=h(\de; g, f; \sembunioil)$, $f_{\de}:=f+h_{\de}(x)f=\sum_{i, j}c'_{ij}x^iy^j$ , $g_{\de}:=g+h'_{\de}(x)f=\sum_{i, j}d'_{ij}x^iy^j$ and $G_{\de}:=g_{\de}-(d'_{10}/c'_{10})f_{\de}=\sum_{i, j}e_{ij}x^iy^j$.
Then $G_{\de}\in \Elim$ (see Definitions \ref{def_G_de} and \ref{H_Elim}), and by the construction of $g_{\de}$, the term $\beta:=d'_{00}-d_{00}\in k$ satisfies $\val(\beta)>0$.
We set
\begin{equation*}
\tilde{d}_{00}=\alpha-\beta+\frac{d'_{10}}{c'_{10}}c'_{00}=\alpha+d_{00}-d'_{00}+\frac{d'_{10}}{c'_{10}}c'_{00}=\alpha+d_{00}-e_{00}.
\end{equation*}
Since we have
\begin{eqnarray*}
\val(\alpha)=\kappa \geq 0,\ \val(\beta)>0,\ \val\left(\frac{d'_{10}}{c'_{10}}c'_{00}\right)=0,
\end{eqnarray*}
we have $\val(\tilde{d}_{00})=0=\val(d_{00})$ if $\kappa>0$.
If $\kappa=0$, we may assume the same by replacing $\alpha$ if neccesary.
Let $g':=g-d_{00}+\tilde{d}_{00}$.
Then, we have $\trop(g')=\trop(g)$.
Note that
\[
h(\de; g', f; \sembunioil)=h(\de; g-d_{00}+\tilde{d}_{00}, f; \sembunioil)=h(\de; g, f; \sembunioil)=h'_{\de},
\]
since in the algorithm of Lemma \ref{lem_f_L}, the coefficient of $g$ at $\bio$ is not used.
For the Laurent polynomial
\[
G'_{\de}:=g'+h'_{\de}(x)f-\frac{d'_{10}}{c'_{10}}f_{\de}=G_{\de}-d_{00}+\tilde{d}_{00}=\sum_{i, j}e'_{ij}x^iy^j,
\]
we have $e'_{00}=\alpha$ and $e'_{\bi}=e_{\bi}$ ($\bi \neq (0, 0)$).
Here, we have $\val(e'_{00})=\val(\alpha)=\kappa$, and hence, by Proposition \ref{prop_d_00}, we have $\trop(V(f, g'))|_L=D|_L$.
\end{proof}

\begin{remark}\label{irekae}
In Corollary \ref{cor_d_00}, we change the coefficient $d_{\bio}$.
By symmetry, we may change the coefficient $d_{\bil}$ instead.
\end{remark}

\begin{corollary}\label{cor}
Let $f$ and $g$ be Laurent polynomials in $k[x^{\pm1}, y^{\pm1}]$, $L\in \cSI(f, g)$ a ray or a line segment and $\Phi_2(L)=\sembunjojl\in \Delta_g$.
Let $D:=\trop(V(f, g))|_L$, and assume that $D\neq 0$ if $L$ is a ray.
Let $g' \in k[x^{\pm1}, y^{\pm1}]$ be a Laurent polynomial such that $\trop(g)=\trop(g')$ and
\[
v_{\bjo+n(\bjl-\bjo)}(g'-g)>v_{\bjo}(g)+n(v_{\bjl}(g)-v_{\bjo}(g))+\dist(D, E|_L)\ (n\in \mathbb{Z}),
\]
where $E$ is the stable intersection divisor of $\Trop(V(f))$ and $\Trop(V(g))$.
Then, we have
\[
\trop(V(f, g'))|_L=\trop(V(f, g))|_L=D.
\]
\end{corollary}

\begin{proof}
We may assume that $f$, $g$ and $L$ satisfy the condition $(\P')$.
Note that since $\trop(g)=\trop(g')$, we have $L\in \cSI(f, g')$ and  $L$ is contained in the edge of $V(\trop(g'))$ corresponding to $\sembunioil\in \Delta_{g'}$, and hence, $f$, $g'$ and $L$ also satisfy the condition $(\P')$.
Since $\min(v_{\bjo+n(\bjl-\bjo)}(g'-g))>\dist(D, E|_L)$, we can take $\de$ such that $\de>\dist(D, E|_L)$ and $\de<v_{\bjo+n(\bjl-\bjo)}(g'-g)$ for any $n\in \mathbb{Z}$.
Let $h_{\de}:=h(\de; f, f; \sembunioil)$ and $h'_{\de}:=h(\de; g, f; \sembunioil)$.
Let $f_{\de}:=f+h_{\de}(x)f=\sum_{i, j}c'_{ij}x^iy^j$, $g_{\de}:=g+h'_{\de}(x)f=\sum_{i, j}d'_{ij}x^iy^j$ and $g'_{\de}:=g'+h'_{\de}(x)f=\sum_{i, j}d''_{ij}x^iy^j$.
Then, we have $g'_{\de}=g'+h'_{\de}(x)f=(g'-g)+g_{\de}$.
Here, by the assumption, we have
\[
v_{\bio+n(\bil-\bio)}(g'-g)>\de\ (n\in \mathbb{Z}).
\]
Combined with Lemma \ref{lem_f_g} (2), this implies that $\mu\left(g'+h'_{\de}(x)f; \sembunioil \right)>\de$, and hence, $(h_{\de}, 0, h'_{\de}, 0)\in H_4(\de; f, g'; L)$.
Let $G_{\de}:=g_{\de}-(d'_{10}/c'_{10})f_{\de}$ and $G'_{\de}:=g'_{\de}-(d''_{10}/c'_{10})f_{\de}$.
Then, we have
\begin{eqnarray*}
G'_{\de}&=&(g'-g)+g_{\de}-\frac{d'_{10}+\coeff_{00}(g'-g)}{c'_{10}}f_{\de}\\
&=&(g'-g)+G_{\de}-\frac{\coeff_{00}(g'-g)}{c'_{10}}f_{\de},
\end{eqnarray*}
and hence, $v_{00}(G'_{\de})=v_{00}(G_{\de})=\dist(D, E|_L)$.
Thus, by Proposition \ref{prop_d_00}, we have
\[
\trop(V(f, g'))|_L=\trop(V(f, g))|_L=D.
\]
\end{proof}

\begin{theorem}\label{thm_main1}
Let $f, g\in k[x^{\pm1}, y^{\pm1}]$ be Laurent polynomials and $D$ a divisor satisfying the condition $(*)$ in Definition \ref{star}.
Let $\cSI'$ be a subset of $\cSI(f, g)$ and write $\cPI:=\cPI(f, g)$.
Assume that $\cSI'$ is acyclic with respect to $\Phi_2$ and that for each $L\in \cSI'$, we have $\dist(D|_L, E|_L)<\mu(g; \Phi_2(L))$.
Then, there exists $g'\in k[x^{\pm1}, y^{\pm1}]$ such that $\trop(g')=\trop(g)$ and
\[
\trop(V(f, g'))|_{\cSI' \cup \cPI}=D|_{\cSI' \cup \cPI}.
\]
\end{theorem}

\begin{proof}
Let $g=\sum_{i, j}d_{ij}x^iy^j$ and $\mathcal{C}$ the union of the elements of $\Delta':=\Phi_2(\cSI')$.
We number and order the endpoints of the elements of $\Delta'$ as $p_1< \dots < p_n$ so that this ordering is normal on each tree of the forest.
We write $L_{ij}\in \cSI'$ for the ray or the line segment corresponding to $\sembunpipj\in \Delta'$.
We will construct $g'=g-\sum_{i=1}^{n}d_{p_i}\mathbf{x}^{p_i}+\sum_{i=1}^{n}\tilde{d}_{p_i}\mathbf{x}^{p_i}$ by determining $g_{j}:=g-\sum_{i=1}^{j}d_{p_i}\mathbf{x}^{p_i}+\sum_{i=1}^{j}\tilde{d}_{p_i}\mathbf{x}^{p_i}$ ($j=1, \dots, n$) inductively.
Assume that we have determined $g_{t-1}$ with $\trop(g)=\trop(g_{t-1})$ and so that $\trop(V(f, g_{t-1}))|_{L}=D|_{L}$ holds for $L\in \cSI'$ if both vertices of $\Phi_2(L)$ belong to $\{p_1, \dots, p_{t-1}\}$.
Let $T$ be the connected component of $\mathcal{C}$ containing $p_t$, and $m=\min \{i\in \mathbb{Z}\ |\ p_i\in T\}$.
If $t=m$, we set $\tilde{d}_{p_t}=d_{p_t}$.
If $t>m$, there is a unique $s$ such that the path $p_mTp_t$ contains $\sembunpspt$.
By the normality of the ordering, $s<t$ holds, and $\tilde{d}_{p_s}$ is already determined.
By the assumption, we have $\dist(D|_{L_{st}}, E|_{L_{st}})<\mu(g; \sembunpspt)=\mu(g_{t-1}; \sembunpspt)$.
By Corollary \ref{cor_d_00} and Remark \ref{irekae}, we determine an element $\tilde{d}_{p_t}\in k$ such that, if we set $g_t=g_{t-1}-d_{p_t}\mathbf{x}^{p_t}+\tilde{d}_{p_t}\mathbf{x}^{p_t}$, then we have $\val(\tilde{d}_{p_t})=\val(d_{p_t})$ and $\trop(V(f, g_t))|_{L_{st}}=D|_{L_{st}}$.
Note that $p_t$ might be contained in $\Aff(\overline{p_qp_r})$ ($q<r<t$, $\overline{p_qp_r}\in \Delta'$).
To show that $\trop(V(f, g_{t}))|_{L_{qr}}=\trop(V(f, g_{t-1}))|_{L_{qr}}$, we check the inequality
\[
v_{p_q+n(p_r-p_q)}(g_t-g_{t-1})>v_{p_q}(g_{t-1})+n(v_{p_r}(g_{t-1})-v_{p_q}(g_{t-1}))+\kappa_{qr},
\]
where $\kappa_{qr}:=\dist(D|_{L_{qr}}, E|_{L_{qr}})$, and apply Corollary \ref{cor}.
This clearly holds for $n=0$ and $1$.
For $n\neq0, 1$, this follows from
\begin{eqnarray*}
&&v_{p_q+n(p_r-p_q)}(g_t-g_{t-1})-v_{p_q}(g_{t-1})-n(v_{p_r}(g_{t-1})+v_{p_q}(g_{t-1}))-\kappa_{qr}\\
&\geq&v_{p_q+n(p_r-p_q)}(g_{t-1})-v_{p_q}(g_{t-1})-n(v_{p_r}(g_{t-1})+v_{p_q}(g_{t-1}))-\kappa_{qr}\\
&\geq&\mu(g_{t-1}; \sembunpqpr)-\kappa_{qr}\\
&>&0.
\end{eqnarray*}

By repeating this process, we get a Laurent polynomial $g'=g-\sum_{i=1}^{n}d_{p_i}\mathbf{x}^{p_i}+\sum_{i=1}^{n}\tilde{d}_{p_i}\mathbf{x}^{p_i}$ such that for all $L\in \cSI'$, we have
\[
\trop(V(f, g'))|_L=D|_L.
\]

Since we have $\trop(g')=\trop(g)$, we have $\cPI(f, g)=\cPI(f, g')\subset \trop(V(f, g'))$ with the multiplicities taken into account by Theorem \ref{OR6.13}.
This concludes the proof of Theorem \ref{thm_main1}.
\end{proof}

\begin{theorem}\label{thm_main2}
Let $f, g\in k[x^{\pm1}, y^{\pm1}]$ be Laurent polynomials and $D$ a divisor satisfying the condition $(*)$ in Definition \ref{star}.
Let $\cSI'$ be a subset of $\cSI(f, g)$ and write $\cPI:=\cPI(f, g)$.
Assume that $\cSI'$ is acyclic with respect to $\Phi_2$ and that we can number and order the endpoints of the elements of $\Delta':=\Phi_2(\cSI')$ as $p_1< \dots < p_n$ so that this order is normal on each tree of the forest and that for each element $\sembunpipj$ of $\Delta'$, its affine span $\Aff(\sembunpipj)$ does not contain a point $p_l$ with $l>i, j$.
Then, there exists $g'\in k[x^{\pm1}, y^{\pm1}]$ such that $\trop(g')=\trop(g)$ and 
\[
\trop(V(f, g'))|_{\cSI' \cup \cPI}=D|_{\cSI' \cup \cPI}.
\]
\end{theorem}

\begin{proof}
Let $g=\sum_{i, j}d_{ij}x^iy^j$ and $\mathcal{C}$ the union of the elements of $\Delta'$.
We write $L_{ij}\in \cSI'$ for the ray or the line segment corresponding to $\sembunpipj\in \Delta'$.
Let us construct $g'=g-\sum_{i=1}^{n}d_{p_i}\mathbf{x}^{p_i}+\sum_{i=1}^{n}\tilde{d}_{p_i}\mathbf{x}^{p_i}$ by determining $g_{t}:=g-\sum_{i=1}^{t}d_{p_i}\mathbf{x}^{p_i}+\sum_{i=1}^{t}\tilde{d}_{p_i}\mathbf{x}^{p_i}$ ($t=1, \dots, n$) inductively, as in the proof of Theorem \ref{thm_main1}.
By the assumption, $p_t$ is not contained in $\Aff(\overline{p_qp_r})$ ($q<r<t$, $\overline{p_qp_r}\in \Delta'$).
Combined with Corollary \ref{cor}, it follows that $\trop(V(f, g_{t}))|_{L_{qr}}=\trop(V(f, g_{r}))|_{L_{qr}}$.
Thus, for all $L\in \cSI'$, we have $\trop(V(f, g'))|_L=D|_L$.

Since we have $\trop(g')=\trop(g)$, we have $\cPI(f, g)=\cPI(f, g')\subset \trop(V(f, g'))$ with the multiplicities taken into account by Theorem \ref{OR6.13}.
Thus, we conclude the proof of Theorem \ref{thm_main2}.
\end{proof}

As an example of applications of Theorem \ref{thm_main2}, we have the following corollary, which deals with the case where a tropical line and a smooth tropical plane curve intersect.

\begin{corollary}
Let $f, g\in k[x^{\pm1}, y^{\pm1}]$ be Laurent polynomials such that $\trop(f)=x\oplus y\oplus 0$ and $\Trop(V(g))$ is smooth.
Let a divisor $D$ satisfy the condition $(*)$ in Definition \ref{star}.
Assume that the origin $(0, 0)$ is not a vertex of $\Trop(V(g))$.
Then, there exists a Laurent polynomial $g'\in k[x^{\pm1}, y^{\pm1}]$ such that $\trop(g')=\trop(g)$ and $\trop(V(f, g'))=D$.
\end{corollary}

\begin{proof}
First, we show that all the connected components of $\Trop(V(f)) \cap \Trop(V(g))$ are in $\cSIfg\cup \cPI(f, g)$.
Let $A$ be a  connected component of $(\Trop(V(f)) \cap \Trop(V(g))) \setminus \cPI(f, g)$.
Since the origin $(0, 0)$ is not a vertex of $\Trop(V(g))$, it is clear that $A$ is either a ray or a line segment.
If the origin is an endpoind of $A$, the origin is a smooth vertex in $\Trop(V(f))$ and is contained in the interior of an edge of $\Trop(V(g))$.
An endpoint $P\neq (0, 0)$ of $A$ is a smooth vertex of $\Trop(V(g))$ and is contained in the interior of an edge of $\Trop(V(f))$.
Therefore, all the multiplicities of the endpoints of $A$ are $1$.
Since $\Trop(V(g))$ is smooth, it is clear that the interior of $A$ does not contain a vertex of $\Trop(g)$.
Therefore, we have $A\in \cSIfg$.

Next, let us show that the map $\Phi_2$ is injective and the union of $\Delta':=\Phi_2(\cSI(f, g))$ is a forest.
First, note that $\Sone(\Trop(V(f)))$ consists of three rays and they have different slopes and that each region of $\mathbb{R}^2\setminus \Trop(V(g))$ is a convex polyhedral set.
Thus, if $\Phi_2(L)=\Phi_2(L')$, then we have $L=L'$.
Thus, the map $\Phi_2$ is injective.
Next, we show that the union of $\Delta'$ is a forest.
Assume that the union of $\Delta'$ is not a forest, i.e., it contains a cycle $C$.
Let $q_{1}, \dots, q_{m}$ ($m\geq3$) be the vertices of $C$ such that $\sembunqiqil \in \Delta'$ for all $i=1, \dots,m$ (we regard $m+1=1$).
Let
\begin{eqnarray*}
D_x&:=&\{(x, y)\in \mathbb{R}^2\ |\ x>y,\ x>0\},\\
D_y&:=&\{(x, y)\in \mathbb{R}^2\ |\ y>x,\ y>0\},\\
D_0&:=&\{(x, y)\in \mathbb{R}^2\ |\ 0>x,\ 0>y\}.
\end{eqnarray*}
Let $\overline{D_{i}}$ $(i=1, \dots, m$) be the closures of the domains of $\mathbb{R}^2\setminus \Trop(V(g))$ corresponding to $q_i$.
Then, $\overline{D_{i}}$ are convex polyhedral sets, and hence, each intersection $\overline{D_{i}}\cap \overline{D_{i+1}}$ is contained in exactly one of the edges of $\Trop(V(f))$.
Assume that $\overline{D_{i}}\cap \overline{D_{i+1}}$ ($i\geq2$) is contained in a ray $Y_-:=\{(0, y)\in \mathbb{R}^2\ |\ y\leq 0\}$ in $\Trop(V(f))$ (we can handle the cases where it is contained in other rays in a similar way).
Then, we have $\overline{D_{i}}\cap D_0\neq \emptyset$ or $\overline{D_{i+1}}\cap D_0\neq \emptyset$.
By renumbering if necessary, assume that $\overline{D_{i}}\cap D_0\neq \emptyset$.
Then, we have $\overline{D_{i}}\cap D_x= \emptyset$, $\overline{D_{i+1}}\cap D_x\neq \emptyset$ and $\overline{D_{i+1}}\cap D_0= \emptyset$.
Here, since $\overline{D_{i+1}}$ is a convex set and intersects $D_x$, the intersection $\overline{D_{i+1}}\cap \overline{D_{i+2}}$ must be contained in the ray $XY:=\{(x, y)\in \mathbb{R}^2\ |\ x=y\geq 0\}$.
By similar arguments, we have $\overline{D_{i-1}}\cap \overline{D_{i}}\subset X_-:=\{(x, 0)\in \mathbb{R}^2\ |\ x\leq 0\}$, and so on.
Thus, if $\sembunqiqil\in \Phi_2(\cSI(f, g))$ is the bold line segment in (a) of Figure \ref{fpfcor}, $\Phi_2(\cSI(f, g))$ must contain the bold line segments in (b) of Figure \ref{fpfcor}.
Here, the $2$-dimensional cell of $\Delta_g$ enclosed by the bold line segments in (b) of Figure \ref{fpfcor} corresponds to a vertex of $\Trop(V(g))$.
Since the edges of $\Trop(V(g))$ corresponding to the three $1$-simplices are contained in the three edges of $\Trop(V(f))$, this vertex must be the origin, and this contradicts the assumption.
Thus, the union of $\Phi_2(\cSI(f, g))$ is a forest.

To prove the statement, it is sufficient to show that we can number and order the endpoints of the elements of $\Delta'$ as $p_1< \dots < p_n$ so that this order is normal on each tree of the forest and that for each element $\sembunpipj$ of $\Delta'$, its affine span $\Aff(\sembunpipj)$ does not contain a point $p_l$ with $l>i, j$.
Note that for each $\sembunij\in \Delta'$, we have $\val(d_{\bi})=\val(d_{\bj})<\val(d_{\bl})$ ($\bl \in (\Aff(\sembunij)\cap \mathbb{Z}^2)\setminus \{\bi, \bj\}$).
Hence, if $\bi'$ and $\bj'$ are contained in the same connected component of the union of $\Delta'$, then $\val(d_{\bi'})=\val(d_{\bj'})$.
Let $p_1, \dots, p_n$ be the endpoints of the elements of $\Delta'$ such that $\val(d_{p_1})\geq \val(d_{p_2})\geq \dots \geq \val(d_{p_n})$ and the order $p_1< \dots < p_n$ is normal on each tree of the forest.
For an element $\sembunpipj$ of $\Delta'$, if its affine span $\Aff(\sembunpipj)$ contains a point $p_l$, then $\val(d_{p_l})>\val(d_{p_i})= \val(d_{p_j})$, and hence, by the condition of the numbering of the endpoints $p_1, \dots, p_n$, we have $l<i, j$.
\end{proof}

\begin{figure}[H]
\centering
\begin{tikzpicture}
\coordinate (L1) at (0,0);
\coordinate (L2) at (2,0);
\coordinate (L3) at (2.5,0);
\coordinate (L4) at (4.5,0);
\coordinate (L5) at (5,0);
\coordinate (L6) at (7,0);
\coordinate (L7) at (7.5,0);
\coordinate (L8) at (9.5,0);
\coordinate (L9) at (0,2);
\coordinate (L10) at (2.5,2);
\coordinate (L11) at (5,2);
\coordinate (L12) at (7.5,2);
\coordinate (L13) at (0.5,0.4);
\coordinate (L14) at (0.9,0.4);
\coordinate (L15) at (3,0.4);
\coordinate (L16) at (3.4,0.4);
\coordinate (L17) at (3,0.8);
\coordinate (L18) at (5.5,0.4);
\coordinate (L19) at (5.9,0.4);
\coordinate (L20) at (5.5,0.8);
\coordinate (L21) at (8,0.4);
\coordinate (L22) at (8.4,0.4);
\coordinate (L23) at (8,0.8);

\draw (L1)--(L2)--(L9)--cycle;
\draw (L3)--(L4)--(L10)--cycle;
\draw [very thick](L13)--(L14);
\draw [very thick](L15)--(L16)--(L17)--cycle;

\coordinate [label=below:\text{(a)}] (a) at (1,-0.25);
\coordinate [label=below:\text{(b)}] (b) at (3.5,-0.25);
\end{tikzpicture}
\caption{Elements of $\Phi_2(\cSI(f, g))$.}
\label{fpfcor}
\end{figure}

\section{Examples}
In the following, let $k=\mathbb{C}\{\!\{t\}\!\}$ be the field of Puiseux series with coefficients in the complex numbers with the usual valuation.

\begin{example}\label{ex3}
Let
\begin{eqnarray*}
&&f=t^3x^2y^2+t^2x^2y+t^2xy^2+xy+x+y+t^{-1}\in k[x^{\pm 1}, y^{\pm 1}],\\
&&g=t^3x^2y^2+xy+x+y\in k[x^{\pm 1}, y^{\pm 1}].
\end{eqnarray*}
Then, the tropical curves $\Trop(V(f))$ and $\Trop(V(g))$ are as in Figure \ref{fex3}, and hence the intersection $\Trop(V(f))\cap \Trop(V(g))$ is the union of the elements of $\cLS(f, g)$.
If we set $\cSI'=\cLSfg$, then it is acyclic with respect to $\Phi_2$ and satisfies the condition in Theorem \ref{thm_main2}.
Therefore, a divisor $D$ satisfying the condition $(*)$ in Definition \ref{star} can be realized.
Here, the edges of $\Trop(V(g))$ corresponding to $\Phi_2(\cSI')$ forms a loop, but this is irrelevant to our condition.
\end{example}

\begin{figure}[H]
\centering
\begin{tikzpicture}
\coordinate (L1) at (-5.5,-1);
\coordinate (L2) at (-5,-1);
\coordinate (L3) at (-5.5,-0.5);
\coordinate (L4) at (-5,-0.5);
\coordinate (L5) at (-4.5,-0.5);
\coordinate (L6) at (-5,0);
\coordinate (L7) at (-4.5,0);
\coordinate (L8) at (-3,-1);
\coordinate (L9) at (-3.5,-0.5);
\coordinate (L10) at (-3,-0.5);
\coordinate (L11) at (-2.5,0);
\draw (L1)--(L2)--(L5)--(L7)--(L6)--(L3)--cycle;
\draw (L2)--(L6);
\draw (L3)--(L5);
\draw (L1)--(L7);
\draw [very thick] (L2)--(L4);
\draw [very thick] (L3)--(L4);
\draw [very thick] (L4)--(L7);
\draw (L8)--(L9)--(L11)--cycle;
\draw [very thick] (L10)--(L8);
\draw [very thick] (L10)--(L9);
\draw [very thick] (L10)--(L11);
\coordinate (A1) at (-0.5,-1);
\coordinate (A2) at (1,-1);
\coordinate (A3) at (-0.5,0.5);
\coordinate (A4) at (0,-1);
\coordinate (A5) at (0.5,-1);
\coordinate (A6) at (-0.5,-0.5);
\coordinate (A7) at (0.5,-0.5);
\coordinate (A8) at (-0.5,0);
\coordinate (A9) at (0,0);
\draw[decorate, decoration={snake, amplitude=0.5pt, segment length=2.5pt}] (A3)--(A1);
\draw[decorate, decoration={snake, amplitude=0.5pt, segment length=2.5pt}] (A1)--(A2);
\draw[decorate, decoration={snake, amplitude=0.5pt, segment length=2.5pt}] (A2)--(A3);
\draw (A4)--(A5)--(A7)--(A9)--(A8)--(A6)--cycle;
\draw [very thick] (A6)--(A8);
\draw [very thick] (A4)--(A5);
\draw [very thick] (A7)--(A9);
\draw[decorate, decoration={snake, amplitude=0.5pt, segment length=2.5pt}] (A1)--(-1,-1.5);
\draw[decorate, decoration={snake, amplitude=0.5pt, segment length=2.5pt}] (A2)--(1.5,-1.25);
\draw[decorate, decoration={snake, amplitude=0.5pt, segment length=2.5pt}] (A3)--(-0.75,1);
\draw (A4)--(0,-1.5);
\draw (A5)--(1,-1.5);
\draw (A6)--(-1,-0.5);
\draw (A7)--(1,-0.5);
\draw (A8)--(-1,0.5);
\draw (A9)--(0,0.5);
\coordinate [label=below:\text{$\Delta_f$}] (a) at (-4.9,-1.75);
\coordinate [label=below:\text{$\Delta_g$}] (b) at (-3,-1.75);
\coordinate [label=below:\text{$\Trop(V(f))$ and $\Trop(V(g))$}] (c) at (0.25,-1.75);
\end{tikzpicture}
\caption{The tropical curves and dual subdivisions in Example \ref{ex3}.}
\label{fex3}
\end{figure}
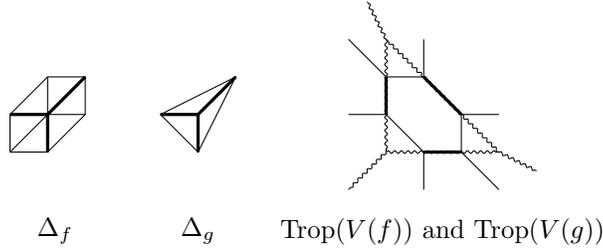

Now we will give two examples to show that we need the acyclicity condition.

\begin{example}\label{ex1}
Let
\begin{eqnarray*}
&&f=xy^3+t^2xy^2+y^3+t^5xy+ty^2+t^5y+t^{10}\in k[x^{\pm 1}, y^{\pm 1}],\\
&&g=ax+by+1\in k[x^{\pm 1}, y^{\pm 1}]\ (\val(a)=\val(b)=0).
\end{eqnarray*}
Then, the tropical curves $\Trop(V(f))$ and $\Trop(V(g))$ are as in Figure \ref{fex1}, and hence the intersection $\Trop(V(f))\cap \Trop(V(g))$ is the union of the elements of $\cLS(f, g)$, and the stable intersection divisor is
\[
E=(0, 0)+(0, -1)+(0, -4)+(0, -5).
\]
Let
\[
D=\left(0, -\frac{1}{4}\right)+\left(0, -\frac{3}{4}\right)+\left(0, -\frac{13}{3}\right)+\left(0, -\frac{14}{3}\right).
\]
Then, it is easy to see that there exists a tropical rational function $\psi$ on $\Trop(V(f))$ satisfying $\Supp(\psi)\subset \Trop(V(f))\cap \Trop(V(g))$ and $(\psi)=D-E$.
Let $L_1=\overline{(0, 0)(0, -1)}$, $L_2=\overline{(0, -4)(0, -5)}$ and $\cSI'=\cLS(f, g)=\{L_1, L_2\}$.
Note that the map $\Phi_2|_{\cSI'}$ is not injective.
Assume that $\trop(V(f, g))|_{\cSI'}=D$.

First, we consider $\trop(V(f, g))|_{L_1}$.
Noting that $\Phi_1(L_1)=\overline{(0, 3)(1, 3)}$ and $\Phi_2(L_1)=\overline{(0, 0)(1, 0)}$, we easily see that $(0, 0, 0, 0)\in \mathrm{H}_4(1; f, g; L_1)$ and that $\sum_{i, j}e_{ij}x^iy^j=g-ay^{-3}f$ belongs to $\mathrm{Elim}(1; f, g; L_1)$, we have $e_{00}=1-a$ and, by Proposition \ref{prop_d_00}, $\val(1-a)=1/4$.

Next, let us consider $\trop(V(f, g))|_{L_2}$.
For $\sum_{i, j}e'_{ij}x^iy^j:=g-(a/t^5)y^{-1}f\in \mathrm{Elim}(1; f, g; L_2)$, we have $e'_{00}=1-a$ and $\val(1-a)=1/3$.
This is a contradiction.
Therefore,  there does not exist $g\in k[x^{\pm1}, y^{\pm1}]$ such that $\trop(g)=x\oplus y\oplus 0$ and $\trop(V(f, g))|_{\cSI'}=D$.
\end{example}

Example \ref{ex1} explains why we need the assumption that the map $\Phi_2|_{\cSI'}$ is injective.

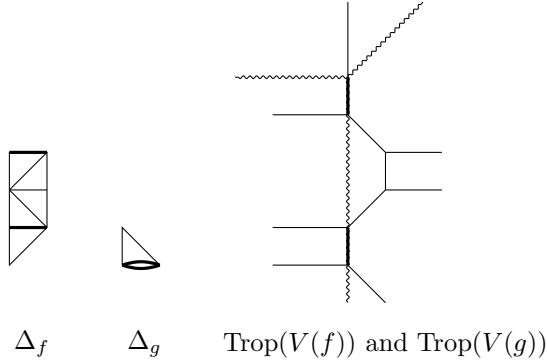
\begin{figure}[H]
\centering
\begin{tikzpicture}
\coordinate (L1) at (-4.5,-1);
\coordinate (L2) at (-4.5,-0.5);
\coordinate (L3) at (-4.5,0);
\coordinate (L4) at (-4.5,0.5);
\coordinate (L5) at (-4,0.5);
\coordinate (L6) at (-4,0);
\coordinate (L7) at (-4,-0.5);
\coordinate (L9) at (-3,-1);
\coordinate (L10) at (-2.5,-1);
\coordinate (L11) at (-3,-0.5);
\draw (L1)--(L2)--(L3)--(L4)--(L5)--(L6)--(L7)--cycle;
\draw (L2)--(L7);
\draw (L3)--(L7);
\draw (L3)--(L6);
\draw (L3)--(L5);
\draw [very thick] (L4)--(L5);
\draw [very thick] (L2)--(L7);
\draw (L9)--(L11)--(L10);
 \draw [very thick] (L9) to [out=18,in=162] (L10);
 \draw [very thick] (L9) to [out=-18,in=198] (L10);
\coordinate (A1) at (0,1);
\coordinate (A2) at (0.5,0.5);
\coordinate (A3) at (0.5,0);
\coordinate (A4) at (0,-0.5);
\coordinate (A5) at (0,-1);
\draw (0,2.5)--(A1)--(A2)--(A3)--(A4)--(A5)--(0.5,-1.5);
\draw [very thick] (A1)--(0,1.5);
\draw [very thick] (A4)--(A5);
\draw[decorate, decoration={snake, amplitude=0.5pt, segment length=2.5pt}] (0,1.5)--(-1.5,1.5);
\draw[decorate, decoration={snake, amplitude=0.5pt, segment length=2.5pt}] (0,1.5)--(0,-1.5);
\draw[decorate, decoration={snake, amplitude=0.5pt, segment length=2.5pt}] (0,1.5)--(1,2.5);
\draw (A1)--(-1,1);
\draw (A2)--(1.25,0.5);
\draw (A3)--(1.25,0);
\draw (A4)--(-1,-0.5);
\draw (A5)--(-1,-1);
\coordinate [label=below:\text{$\Delta_f$}] (a) at (-4.2,-1.75);
\coordinate [label=below:\text{$\Delta_g$}] (b) at (-2.7,-1.75);
\coordinate [label=below:\text{$\Trop(V(f))$ and $\Trop(V(g))$}] (c) at (0.5,-1.75);
\end{tikzpicture}
\caption{The tropical curves and dual subdivisions in Example \ref{ex1}.}
\label{fex1}
\end{figure}

\begin{remark}
If we regard the two bold line segments in $\Delta_g$ as different things as in Figure \ref{fex1}, they form a cycle.
Thus, we can regard the assumption that the map $\Phi_2|_{\cSI'}$ is injective is a part of the assumption that the union of the elements of $\Phi_2(\cSI')$, regarded as a multiset, is a forest.
\end{remark}

\begin{example}\label{ex2}
Let
\begin{eqnarray*}
&&f=t^3x^3y^3+tx^3y^2+tx^2y^3+x^2y^2+tx^2y+txy^2+txy+t^3\in k[x^{\pm 1}, y^{\pm 1}],\\
&&g=ax+by+1\in k[x^{\pm 1}, y^{\pm 1}]\ (\val(a)=\val(b)=0).
\end{eqnarray*}
Then, the tropical curves $\Trop(V(f))$ and $\Trop(V(g))$ are as in Figure \ref{fex2}, and hence the intersection $\Trop(V(f))\cap \Trop(V(g))$ is the union of the elements of $\cLS(f, g)$, and the stable intersection divisor is
\[
E=(-2, 0)+(-1, 0)+(0, -2)+(0, -1)+(1, 1)+(2, 2).
\]
Let
\[
D=\left(-\frac{7}{4}, 0\right)+\left(-\frac{5}{4}, 0\right)+\left(0, -\frac{5}{3}\right)+\left(0, -\frac{4}{3}\right)+\left(\frac{4}{3}, \frac{4}{3}\right)+\left(\frac{5}{3}, \frac{5}{3}\right).
\]
It is easy to see that there exists a tropical rational function $\psi$ on $\Trop(V(f))$ satisfying $\Supp(\psi)\subset \Trop(V(f))\cap \Trop(V(g))$ and $(\psi)=D-E$.
Let $L_1=\overline{(1, 1)(2, 2)}$, $L_2=\overline{(-1, 0)(-2, 0)}$, $L_3=\overline{(0, -1)(0, -2)}$ and $\cSI'=\cLS(f, g)=\{L_1, L_2, L_3\}$.
Note that the union of the elements of $\Phi_2(\cLS')$ is not a forest.
Assume that $\trop(V(f, g))|_{\cLS'}=D$.

First, we consider $\trop(V(f, g))|_{L_1}$.
We have $\Phi_2(L_1)=\overline{(0, 1)(1, 0)}$.
We regard $(0, 1)$ as $\bjo$, and then for $\sum_{i, j}e_{ij}x^iy^j:=g-(a/t)x^{-2}y^{-2}f\in \mathrm{Elim}(1; f, g; L_1)$, we have $e_{01}=b-a$.
Then, by Proposition \ref{prop_d_00}, we have $\val(b-a)=1/3$.

Next, let us consider $\trop(V(f, g))|_{L_2}$ and $\trop(V(f, g))|_{L_3}$.
For $\sum_{i, j}e'_{ij}x^iy^j:=g-(b/t)x^{-1}y^{-1}f\in \mathrm{Elim}(1; f, g; L_2)$, we have $e'_{00}=1-b$ and $\val(1-b)=1/4$.
For $\sum_{i, j}e''_{ij}x^iy^j:=g-(a/t)x^{-1}y^{-1}f\in \mathrm{Elim}(1; f, g; L_3)$, we have $e''_{00}=1-a$ and $\val(1-a)=1/3$.
Thus, we have
\begin{eqnarray*}
&&\val(1-a)=\val(b-a)=\frac{1}{3},\\
&&\val(1-b)=\frac{1}{4}.
\end{eqnarray*}
Then, we would have
\[
\frac{1}{3}=\val(1-a)=\val((1-b)+(b-a))=\frac{1}{4}.
\]
This is a contradiction.
Therefore,  there does not exist $g\in k[x^{\pm1}, y^{\pm1}]$ such that $\trop(g)=x\oplus y\oplus 0$ and $\trop(V(f, g))|_{\cSI'}=D$.
\end{example}

Example \ref{ex2} explains why we need the assumption that the union of the elements of $\Phi_2(\cSI')$ is a forest.

\begin{figure}[H]
\centering
\begin{tikzpicture}
\coordinate (L1) at (-5.5,-1);
\coordinate (L2) at (-5,-0.5);
\coordinate (L3) at (-4.5,-0.5);
\coordinate (L4) at (-5,0);
\coordinate (L5) at (-4.5,0);
\coordinate (L6) at (-4,0);
\coordinate (L7) at (-4.5,0.5);
\coordinate (L8) at (-4,0.5);
\coordinate (L9) at (-3,-1);
\coordinate (L10) at (-2.5,-1);
\coordinate (L11) at (-3,-0.5);
\draw (L1)--(L3)--(L6)--(L8)--(L7)--(L4)--cycle;
\draw (L2)--(L5);
\draw (L4)--(L6);
\draw (L3)--(L7);
\draw [very thick] (L6)--(L7);
\draw (L1)--(L2);
\draw [very thick] (L2)--(L3);
\draw [very thick] (L2)--(L4);
\draw [very thick](L9)--(L10)--(L11)--cycle;
\coordinate (A1) at (-1,0);
\coordinate (A2) at (0,-1);
\coordinate (A3) at (-0.5,0);
\coordinate (A4) at (0,-0.5);
\coordinate (A5) at (-0.5,0.5);
\coordinate (A6) at (0.5,0.5);
\coordinate (A7) at (0.5,-0.5);
\coordinate (A8) at (1,1);
\draw (A3)--(A4)--(A7)--(A6)--(A5)--cycle;
\draw [very thick] (A6)--(A8);
\draw [very thick] (A1)--(A3);
\draw [very thick] (A2)--(A4);
\draw (A1)--(A2);
\draw[decorate, decoration={snake, amplitude=0.5pt, segment length=2.5pt}] (0,0)--(-1.5,0);
\draw[decorate, decoration={snake, amplitude=0.5pt, segment length=2.5pt}] (0,0)--(0,-1.5);
\draw[decorate, decoration={snake, amplitude=0.5pt, segment length=2.5pt}] (0,0)--(1.5,1.5);
\draw (1,1.5)--(A8)--(1.5,1);
\draw (A5)--(-1,1);
\draw (A7)--(1,-1);
\draw (A1)--(-1.5,0.25);
\draw (A2)--(0.25,-1.5);
\coordinate [label=below:\text{$\Delta_f$}] (a) at (-4.5,-1.75);
\coordinate [label=below:\text{$\Delta_g$}] (b) at (-2.7,-1.75);
\coordinate [label=below:\text{$\Trop(V(f))$ and $\Trop(V(g))$}] (c) at (0.5,-1.75);
\end{tikzpicture}
\caption{The tropical curves and dual subdivisions in Example \ref{ex2}.}
\label{fex2}
\end{figure}

\end{document}